\documentclass[final,a4paper]{siamltex}
\usepackage{amsmath,amssymb,amsfonts}
\usepackage{color}
\usepackage{verbatim}
\usepackage{eepic,epic,epsfig}
\usepackage[]{graphics}
\usepackage{graphicx,inputenc}
\usepackage{subfigure}
\usepackage{graphicx,subfig}
\graphicspath{{./pics/}}
\usepackage{bm,color}
\usepackage{epstopdf}
\usepackage[]{graphics}

\newtheorem{remark}{Remark}[section]

\newcommand{\al}{\alpha}
\newcommand{\fy}{\varphi}
\newcommand{\la}{\lambda}
\newcommand{\ep}{\epsilon}

\def\tribar{\vert\thickspace\!\!\vert\thickspace\!\!\vert}

\def\dH#1{\dot H^{#1}(\Omega)}

\def\Dal{{^C\kern-1mm\partial_t^\al}}

\def\Om{\Omega}
\def\II{(\Om)}

\def\bPtau{\bar\partial_\tau}
\newcount\icount

\usepackage{showkeys}

\title%[Fully Discrete Schemes for FDEs]
{Two Fully Discrete Schemes for Fractional Diffusion and Diffusion-Wave
Equations with Nonsmooth Data}
\author {Bangti Jin\thanks{Department of Computer Science, University College London, Gower Street,
London, WC1E 6BT, UK (\texttt{b.jin@ucl.ac.uk, bangti.jin@gmail.com})}
\and Raytcho Lazarov\thanks{Department of Mathematics, Texas A\&M University, College Station, TX 77843-3368, USA
({\texttt{lazarov, zzhou@math.tamu.edu}})} \and Zhi Zhou\footnotemark[2] \thanks{Department of Applied Mathematics and Applied Physics, Columbia University,  500 W. 120th St., New York, NY 10027, USA (\texttt{zhizhou0125@gmail.com})}}

\begin{document}
\maketitle

\begin{abstract}
We consider initial/boundary value problems for the subdiffusion and diffusion-wave
equations involving a Caputo fractional derivative in time. We develop two fully
discrete schemes based on the piecewise linear Galerkin finite element method in space and convolution quadrature
in time with the generating function given by the backward Euler
method/second-order backward difference method, and establish error estimates optimal with
respect to the regularity of problem data. These two schemes are first and second-order
accurate in time for both smooth and nonsmooth data. Extensive numerical experiments
for two-dimensional problems confirm the convergence analysis and robustness of the schemes
with respect to data regularity.
\end{abstract}

\section{Introduction}

\subsection{Mathematical model}
In this work, we develop robust numerical schemes for the subdiffusion and diffusion wave equations.
Let $\Omega\subset \mathbb{R}^d$ ($d=1,2,3$) be a bounded convex polygonal domain with a boundary $\partial\Omega$,
and $T>0$ be a fixed time. Then the mathematical model is given by
\begin{equation}\label{eqn:fde}
  \Dal u(x,t) - \Delta u(x,t) = f(x,t)\quad (x,t)\in\Omega\times (0,T),
\end{equation}
where $f$ is a given source term. Here  $\Dal u $ denotes the Caputo fractional derivative with
respect to time $t$ of order $\alpha$, and it is defined by \cite[pp. 91, eq. (2.4.1)]{KilbasSrivastavaTrujillo:2006}
\begin{equation}\label{eqn:Caputo}
  \Dal u (t) = \frac{1}{\Gamma(n-\alpha)}\int_0^t (t-s)^{n-\alpha-1}
    \frac{d^n}{ds^n}u(s)ds, \quad n-1<\alpha<n, \, n\in\mathbb{N},
\end{equation}
where $\Gamma$ is the Gamma function. We shall also need the Riemann-Liouville
fractional derivative $\partial_t^{\alpha} u$ defined by \cite[pp. 70, eq. (2.1.5)]{KilbasSrivastavaTrujillo:2006}
\begin{equation}\label{eqn:RL}
\partial_t^\alpha u=
\frac{1}{\Gamma(n-\alpha)} \frac{d^n}{dt^n} \int_0^t(t-s)^{n -\alpha-1}u(s)ds, \quad n-1<\alpha<n, \, n\in\mathbb{N}.
\end{equation}
Note that if $\alpha=1$ and $\alpha=2$, then equation \eqref{eqn:fde} represents a parabolic and a hyperbolic equation,
respectively. In this paper we focus on the fractional cases $0<\alpha<1$ and $1<\alpha<2$, with a Caputo derivative,
which are known as the subdiffusion and diffusion-wave equation, respectively, in the literature. In analogy with Brownian motion
for normal diffusion, the model \eqref{eqn:fde} with $0<\alpha<1$ is the macroscopic counterpart of continuous time random walk
\cite{MontrollWeiss:1965,GorenfloMainardiMorettiParadisi:2002}.

Throughout, equation \eqref{eqn:fde} is subjected to the following boundary condition
\begin{equation*}
  \begin{aligned}
    u(x,t) &= 0,\quad (x,t)\in \partial\Omega\times(0,T),
  \end{aligned}
\end{equation*}
and  the initial condition(s)
\begin{equation*}
  \begin{aligned}
    u(x,0) & = v(x), \quad x\in \Omega, \quad \mbox{if} \quad 0<\alpha<1,\\
    u(x,0)  &= v(x), \quad  \partial_t u (x,0) = b(x), \quad x\in \Omega, \ \mbox{if } 1<\alpha<2.
  \end{aligned}
\end{equation*}

The model \eqref{eqn:fde} has received much attention over the last few decades, since
it can adequately capture the dynamics of physical processes involving anomalous transport mechanism.
For example, the subdiffusion equation has been successfully used to describe thermal diffusion in media with fractal
geometry \cite{Nigmatulin:1986}, and highly heterogeneous aquifers \cite{AdamsGelhar:1992,HatanoHatano:1998}.
The diffusion wave equation can be used to model the propagation of
mechanical waves in viscoelastic media \cite{Mainardi:1996,Mainardi:2010}. Further, we refer to
\cite{KilbasSrivastavaTrujillo:2006,Podlubny:1999,Diethelm:2004} for physical applications and
mathematical theory, and \cite{BaleanuDiethelmScalasTrujillo:2012} for a comprehensive survey on
numerical methods for fractional ordinary differential equations.

\subsection{Review of existing schemes}

Due to the excellent modeling capability of problem \eqref{eqn:fde}, its accurate numerical
treatment has been the subject of numerous studies.
A number of efficient schemes, notably based on finite difference in space and various
discretizations in time, have been developed. Their error analysis is often based on Taylor
expansion and the error bounds are expressed in terms of solution smoothness.
The analysis of many  existing methods for problem \eqref{eqn:fde}
requires that the solution be a $C^2$ or $C^3$-function in time, cf. Table \ref{tab:existingscheme}, where
$\tau$ is the time step size; see Section \ref{ssec:existing} for further details. However, such an assumption
is not always valid since often the solution does not have the requisite regularity. For example, the $C^2$ in time
regularity assumption on the solution does not hold for the homogeneous subdiffusion problem.
Specifically, for initial data $v\in L^2(\Omega)$, $f=0$ and $\alpha\in(0,1)$, the following
estimate holds true \cite[Theorem 2.1]{SakamotoYamamoto:2011}:
\begin{equation}\label{eqn:sol-sing}
  \|{\Dal } u(t)\|_{L^2(\Omega)} \leq ct^{-\alpha} \|v\|_{L^2(\Omega)}.
\end{equation}
This shows that the $\alpha$-th order Caputo derivative in time is already unbounded near $t=0$.
Hence, the convergence rates listed in Table \ref{tab:existingscheme} may not hold for nonsmooth data.
Actually, the limited solution regularity underlies the main technical difficulty in developing
robust numerical schemes and in carrying out a rigorous error analysis.

\begin{table}[hbt!]
   \centering
   \caption{The convergence rates for existing schemes for subdiffusion ($0<\alpha<1$).
   In the table, RL denotes the Riemann-Liouville fractional derivative, $\tau$ is
   the time step size, and $\bar u$
   is the zero extension in time of $u$ to $\mathbb{R}$.}
   \label{tab:existingscheme}
   \vspace{-0.3cm}
   \begin{tabular}{llll}
     \hline
     method &  rate & derivative & regularity assumption\\
     \hline
     Lin-Xu \cite{LinXu:2007} & $O(\tau^{2-\alpha})$ & Caputo & $\forall x \in \Omega$, $u$ is $C^2$ in $t$ \\
     Zeng et al I \cite{ZengLiLiuTurner:2013} & $O(\tau^{2-\alpha})$ & Caputo & $\forall x \in \Omega$, $u$ is $C^2$ in $t$\\
     Zeng et al II \cite{ZengLiLiuTurner:2013}& $O(\tau^{2-\alpha})$ & Caputo & $\forall x \in \Omega$, $u$ is $C^2$ in $t$\\
     Li-Xu \cite{LiXu:2010}   & $O(\tau^2)$ & Caputo & $\forall x \in \Omega$, $u$ is $C^3$ in $t$\\
     Gao et al \cite{GaoSunZhang:2014}        & $O(\tau^{3-\alpha})$ & Caputo & $\forall x \in \Omega$, $u$ is $C^3$ in $t$\\
     L1 scheme \cite{OldhamSpanier:1974,LanglandsHenry:2005} & $O(\tau^{2-\alpha})$ & RL & $\forall x \in \Omega$, $u$ is $C^2$ in $t$\\
     SBD \cite{LiDing:2014}             & $O(\tau^2)$ & RL & $\forall x \in \Omega$, ${_{-\infty}^RD_t^{3-\alpha}} \bar u$ is $L^1$ in $t$\\
     \hline
   \end{tabular}
\end{table}

In the subdiffusion case, there are two predominant discretization techniques (in time): the L1-type approximation
\cite{OldhamSpanier:1974,LanglandsHenry:2005,LinXu:2007,SunWu:2006,LiXu:2009,LinLiXu:2011,GaoSunZhang:2014}
and the Gr\"{u}nwald-Letnikov approximation \cite{YusteAcedo:2005,ChenLiuTurnerAnh:2007,ZengLiLiuTurner:2013}; see
Table \ref{tab:existingscheme} for a summary.
To the first group belongs the method devised by
Langlands and Henry \cite{LanglandsHenry:2005}. They analyzed the discretization error for
the Riemann-Liouville derivative. Also, Lin and Xu \cite{LinXu:2007}
(see also \cite[pp. 140]{OldhamSpanier:1974}) developed the L1 scheme (of finite
difference nature) for the Caputo derivative and a Legendre spectral method in space, and analyzed
the stability and convergence rates. It has a local truncation error $O(\tau^{2-\alpha})$.
Li and Xu \cite{LiXu:2009} developed a space-time
spectral element method, but only for zero initial data; see also \cite{LinLiXu:2011,FordXiaoYan:2011}.
In \cite{LiXu:2010} a variant of the L1 approximation was analyzed, and a convergence rate $O(\tau^2)$ was
established for $C^3$ solutions. Recently, a new L1-type formula based on
quadratic interpolation was derived in \cite{GaoSunZhang:2014} with a convergence rate
$O(\tau^{3-\alpha})$ for smooth solutions.
We also refer readers to \cite{McLeanMustapha:2009,MustaphaMcLean:2013,MustaphaSchotzau:2014}
for studies on discontinuous Galerkin discretization
of the Riemann-Liouville derivative in time, and \cite{ChenXuHesthaven:2015,ChenShenWang:2015}
for spectral methods, which merits exponential convergence for smooth solutions.

In the second group, Yuste and Acedo \cite{YusteAcedo:2005} suggested a Gr\"{u}nwald-Letnikov
discretization of the Riemann-Liouville derivative and central finite difference in space, and provided a von
Neumann type stability analysis. Zeng et al \cite{ZengLiLiuTurner:2013} developed two numerical
schemes of the order $O(\tau^{2-\alpha})$ based on an integral reformulation of problem \eqref{eqn:fde},
a fractional linear multistep method in time and finite element method (FEM) in space,
and analyzed their stability and convergence.
However, the schemes are not robust with respect to data regularity; see Remark \ref{rmk:Zeng} and the comparative study in Section
\ref{sec:numeric}. Convolution quadrature \cite{Lubich:1986,Lubich:1988} provides a systematic framework for
deriving high-order schemes for the Riemann-Liouville derivative, and has been the foundation of many works
(see e.g., \cite{YusteAcedo:2005,Yuste:2006,ChenLiuTurnerAnh:2007} for some early works). However, the error
estimates in these works were derived under the assumption that the solution is sufficiently smooth. Further,
 all works, except \cite{ZengLiLiuTurner:2013}, focus exclusively on the Riemann-Liouville derivative;
and high-order methods were scarcely applied, despite that they can be conveniently analyzed
even for nonsmooth data \cite{LubichSloanThomee:1996,CuestaLubichPlencia:2006}.

The study on the diffusion wave equation is scarce. In \cite{SunWu:2006}, a Crank-Nicolson scheme
was developed and its stability and convergence rate were shown. With $b=0$ and  $f=0$, under suitable
regularity assumptions, problem \eqref{eqn:fde} can be rewritten as
\begin{equation*}
  \partial_t u = \frac{1}{\Gamma(\alpha-1)}\int_0^t (t-s)^{\alpha-2}\Delta u(s)ds,
\end{equation*}
with an initial condition $u(0)=v$. This model has been intensively studied \cite{LubichSloanThomee:1996,
McLeanThomee:2010a,CuestaLubichPlencia:2006}, where convolution quadrature and Laplace transform
method were analyzed. The error estimates derived in these works \cite{LubichSloanThomee:1996,
McLeanThomee:2010a,CuestaLubichPlencia:2006} cover the nonsmooth case. This model is closely
connected to \eqref{eqn:fde}, but these problems have different smoothing properties in the inhomogeneous
case \cite{McLean:2010,CuestaLubichPlencia:2006,SakamotoYamamoto:2011}. To the best of our knowledge,
convolution quadrature for \eqref{eqn:fde} with a Caputo derivative and $1<\alpha<2$ has not been studied.

The excessive smoothness required in existing error analysis and the lack of convolution quadrature type schemes for
the diffusion-wave equation with a Caputo derivative motivate us to revisit these issues. The goal of
this work is to develop robust schemes based on convolution quadrature for the model \eqref{eqn:fde} and to derive
optimal error bounds that are expressed in terms of problem data, including nonsmooth data, e.g.,
$v \in L^2(\Omega)$, which is
important in inverse problems and optimal control \cite{JinRundell:2012,JinRundell:2015}.

\subsection{Contributions and organization of the paper}

In this work, we shall develop two fully discrete schemes  for problem \eqref{eqn:fde} based on convolution
quadrature in time generated by backward Euler or second-order backward difference and the piecewise linear
Galerkin FEM in space. This is achieved by reformulating problem \eqref{eqn:fde} using a Riemann-Liouville
derivative. To the best of our knowledge, our application of convolution quadrature to the Caputo derivative
is new, especially for the diffusion wave case, and for the first time a second-order scheme is obtained
for problem \eqref{eqn:fde} for both smooth and nonsmooth data. We shall establish optimal convergence
rates in either case. This is in sharp contrast with existing works on the model \eqref{eqn:fde}, where
the convergence analysis is mostly done under unverifiable solution regularity assumptions.
The error analysis exploits an operator trick \cite{FujitaSuzuki:1991} and a general strategy developed in
\cite{CuestaLubichPlencia:2006}. Extensive two-dimensional (in space) experiments confirm the convergence theory
and their robustness with respect to data regularity.

To illustrate the features of our schemes, we describe one result from Theorem
\ref{thm:err-homo-fully-BE}: for $0<\alpha<1$, $v\in L^2(\Omega)$, $v_h=P_hv$ (with $P_h$ being the $L^2$
projection) and $f=0$, for $n \ge 1$, the fully discrete approximation $U_h^n$ obtained by the backward Euler
method (with a mesh size $h$ and time step size $\tau$) satisfies the following error bound
\begin{equation}\label{eqn:estBE-sub}
   \| u(t_n)-U_h^n \|_{L^2(\Om)} \le c (t_n^{-1}\tau  + t_n^{-\al}h^2)  \| v\|_{L^2\II}.
\end{equation}
This estimate differs from those listed in Table \ref{tab:existingscheme} in several aspects:
\begin{itemize}
  \item[(a)] For any fixed $t_n$, the scheme is first-order accurate in time.
  \item[(b)] It deteriorates near $t=0$, whereas that
  in Table \ref{tab:existingscheme} are uniform in $t$. The factor $t_n^{-1}$
  reflects the singularity behavior \eqref{eqn:sol-sing} for initial data $v\in L^2\II$.
  \item[(c)] The scheme is robust with respect to the regularity of the initial data $v$ in the sense that
  for fixed $t_n$, the first-order in time and second-order in space convergence rates
  hold for both smooth and nonsmooth data.
\end{itemize}

The rest of the paper is organized as follows. In Section \ref{sec:scheme}, we develop two fully
discrete schemes using the Galerkin FEM in space and convolution quadrature in time.
The error analysis of the schemes is given in Section \ref{sec:analysis}. In Section \ref{sec:numeric},
we present extensive numerical experiments to illustrate the convergence behavior of the methods.
A comparison with existing methods is also included. In Appendix \ref{app:diffwave}, we collect
the solution theory for the diffusion wave equation. Throughout, the notation $c$ denotes a
generic constant, which may differ at different occurrences, but it is always
independent of the solution $u$, the mesh size $h$ and the time step size $\tau$.

\section{Fully discrete schemes} \label{sec:scheme}

In this part, we develop two fully discrete schemes, using the standard Galerkin
FEM in space and convolution quadrature in time.
\subsection{Space semidiscrete Galerkin FEM}

Let $\mathcal{T}_h$ be a shape regular and quasi-uniform triangulation of the
domain $\Omega $ into $d$-simplexes, denoted by $T$ and called finite elements. Then
over the triangulation $\mathcal{T}_h$, we define a continuous piecewise linear finite
element space $X_h$ by
\begin{equation*}
  X_h= \left\{v_h\in H_0^1(\Omega):\ v_h|_T \mbox{ is a linear function},\ \forall T \in \mathcal{T}_h\right\}.
\end{equation*}

On the space $X_h$ we define the $L^2$-orthogonal projection $P_h:L^2(\Omega)\to X_h$ and
the Ritz projection $R_h:H^1_0(\Omega)\to X_h$, respectively, by
\begin{equation*}
  \begin{aligned}
    (P_h \fy,\chi) & =(\fy,\chi) \quad\forall \chi\in X_h,\\
    (\nabla R_h \fy,\nabla\chi) & =(\nabla \fy,\nabla\chi) \quad \forall \chi\in X_h,
  \end{aligned}
\end{equation*}
where $(\cdot,\cdot)$ denotes the $L^2\II$-inner product.
The semidiscrete Galerkin scheme for problem \eqref{eqn:fde} reads: find $u_h(t)\in X_h$ such that
\begin{equation}\label{eqn:fem}
  (\Dal u_h,\chi) + a(u_h,\chi) = (f,\chi)\quad\forall \chi\in X_h,
\end{equation}
where the bilinear form $a(u,\chi)$ and the initial data are given by
$ a(u,\chi) = (\nabla u,\nabla \chi),$ and
$u_h(0) = v_h$ and  if $1<\alpha<2$, also $\partial_tu_h(0) = b_h$.
Here $v_h,b_h\in X_h$ are approximations to the initial data $v$ and $b$, respectively.
Following \cite{Thomee:2006}, we choose $v_h\in X_h$ (and similarly for $b_h\in X_h$) depending on
the smoothness of the data: $v_h=R_hv$ if  $v\in \dH 2$ and $v_h=P_hv$ if $v\in L^2\II$.

Upon introducing the discrete Laplacian $\Delta_h: X_h\to X_h$ defined by
\begin{equation}\label{eqn:Delh}
  -(\Delta_h\fy,\chi)=(\nabla\fy,\nabla\chi)\quad\forall\fy,\,\chi\in X_h,
\end{equation}
$f_h(t)= P_h f(t)$, and $A_h=-\Delta_h$, the semidiscrete scheme \eqref{eqn:fem} can be rewritten into
\begin{equation}\label{eqn:fdesemidis}
  \Dal u_h(t) +A_h u_h(t) = f_h(t), \,\, t>0
\end{equation}
with $u_h(0)=v_h\in X_h$, and if $1<\alpha<2$, $\partial_tu_h(0)=b_h\in X_h$.

\begin{remark}
Upon minor modifications, our discussions extend to more general sectorial operators,
including a strongly elliptic second-order differential operator
$A u = -\nabla\cdot(a(x)\nabla u)$, where the conductivity tensor $a(x):\mathbb{R}^d\to \mathbb{R}^{d\times d}$
is smooth and has a positive minimum eigenvalue $\lambda_{\min}(a(x))\geq c_0$ for some $c_0>0$ almost everywhere, and
the Riemann-Liouville derivative operator of order $\beta\in(1,2)$ (with a zero Dirichlet boundary condition)
\cite{JinLazarovPasciakZhou:2014siam}. Further, it is trivial to extend
the fully discrete schemes to the multi-term subdiffusion/diffusion-wave problem.
\end{remark}

\subsection{Fully discrete schemes}\label{ss:fully_discrete}
Now we develop two fully discrete schemes for problem \eqref{eqn:fde}. This is achieved by first
reformulating problem \eqref{eqn:fdesemidis} with a Riemann-Liouville derivative $\partial_t^\alpha$ and then
applying convolution quadrature. Specifically, we rewrite the semidiscrete problem \eqref{eqn:fdesemidis}
using the defining relation of the Caputo derivative. Namely, for $n-1<\alpha<n$, there holds \cite[pp. 91,
equation (2.4.10)]{KilbasSrivastavaTrujillo:2006}
\begin{equation*}
  \Dal \fy(t) : = {\partial_t^\alpha}\left[\fy(t) -\sum_{k=0}^{n-1}\frac{\fy^{(k)}(0)}{k!}t^k\right].
\end{equation*}
In particular, for subdiffusion, $0<\alpha<1$, $\Dal\fy=\partial_t^\alpha(\fy(t)-\fy(0))$, and diffusion-wave, $1<\alpha<2$,
$\Dal \fy (t) = \partial_t^\alpha(\fy(t)-\fy(0)-t\fy'(0))$.
Hence, for $t >0$, the semidiscrete scheme \eqref{eqn:fdesemidis} can be respectively recast into
\begin{eqnarray}
   && {\partial_t^\al} (u_h-v_h) + A_h u_h = f_h, \quad \mbox{for} \quad 0<\alpha<1\label{semi1}\\
   &&{\partial_t^\al} (u_h-v_h-tb_h) + A_h u_h = f_h, \quad \mbox{for} \quad 1<\alpha<2,\label{semi2}
\end{eqnarray}
where $ f_h = P_hf(t)$.
The formulae \eqref{semi1} and \eqref{semi2} form the basis for
time discretization, which is done in the elegant framework -- convolution quadrature -- developed in \cite[Sections 2 and
3]{CuestaLubichPlencia:2006}, initiated in \cite{Lubich:1986,Lubich:1988}.
Below we describe this framework.

Let $K$ be a complex valued or operator valued function which is analytic in
a sector $\Sigma_\theta:=\{ z\in\mathbb{C}: |\arg z|\leq \theta \}$, $\theta\in (\pi/2,\pi)$ and bounded by
\begin{equation}\label{eqn:Kbound}
    \| K(z)  \|\le M|z|^{-\mu}\quad   \forall z\in \Sigma_\theta,
\end{equation}
for some $\mu$, $M\in \mathbb{R}$.  Then $K(z)$ is the Laplace transform of a distribution $k$ on
the real line, which vanishes for $t<0$, has its singular support empty or concentrated at $t=0$,
and which is an analytic function for $t>0$. For $t>0$, the analytic function $k(t)$ is given by the inversion formula
\begin{equation*}
  k(t) = \frac{1}{2\pi\mathrm{i}}\int_\Gamma K(z)e^{zt}dz, \ \ t>0,
\end{equation*}
where $\Gamma$ is a contour lying in the sector of analyticity, parallel to its boundary and oriented with an
increasing imaginary part. With $\partial_t$ being time differentiation, we define $K(\partial_t)$ as the
operator of (distributional) convolution with the kernel $k:K(\partial_t)g=k\ast g$ for a function $g(t)$
with suitable smoothness. Further, the convolution rule of Laplace transform
gives the following associativity property: for the operators $K_1$ and $K_2$
(generated by the kernels $k_1$ and $k_2$), we have
\begin{equation}\label{eqn:semig-contin}
    K_1(\partial_t)K_2(\partial_t)=(K_1K_2)(\partial_t).
\end{equation}

For time discretization, we divide the interval $[0,T]$ into a uniform grid with a time step size
$\tau = T/N$, $N\in\mathbb{N}$ and $0=t_0<t_1<\ldots<t_N=T$, $t_n=n\tau$, $n=0,\ldots,N$.
Then the convolution quadrature $ K(\bPtau) g(t)$ of $ K(\partial_t)g(t)$ is given by
\cite{Lubich:1988}
\begin{equation}\label{convol-quadrature}
  K(\bPtau) g(t) = \sum_{0\leq j\tau\leq t}\omega_jg(t-j\tau), \ \ t>0,
\end{equation}
where the quadrature weights $\{\omega_j\}_{j=0}^\infty$ are determined by
$\sum_{j=0}^\infty \omega_j\xi^j = K(\delta(\xi)/\tau). $
Here $\delta$ is the quotient of the generating polynomials of a stable and consistent linear
multistep method \cite[Chapter 3]{HairerNorsettWanner:1993}. In this work, we consider
the backward Euler (BE) method and second-order backward difference (SBD) method, for which
\begin{equation*}
  \delta(\xi) = \left\{\begin{aligned}
    &(1-\xi), \qquad &\ \  \mbox{BE},\\
    &(1-\xi) + (1-\xi)^2/2, &\ \ \mbox{SBD}.
  \end{aligned}\right.
\end{equation*}
The weights $\{\omega_j\}$ can be computed efficiently via fast Fourier transform \cite{Podlubny:1999}.
The associativity property is also valid for convolution quadrature:
\begin{equation}\label{eqn:semig-disc}
     K_1(\bPtau)K_2(\bPtau)=(K_1K_2)(\bPtau)\quad \mbox{and}\quad
 K_1(\bPtau) (k\ast g) = (K_1(\bPtau)k)\ast g.
\end{equation}

Now we are ready to derive fully discrete schemes. Hence we rewrite
the scheme \eqref{semi1} in a convolution form
\begin{equation}\label{sol-semi}
 u_h(t) = (\partial_t^\al + A_h)^{-1} \partial_t^\al v_h+ (\partial_t^\al + A_h)^{-1}  f_h.
\end{equation}
Then the associativity property \eqref{eqn:semig-disc} yields the BE scheme for the case $0<\alpha <1$:
\begin{equation}\label{sol-be}
U_h^n = (\bPtau^\al + A_h)^{-1} \bPtau^\al v_h+ (\bPtau^\al + A_h)^{-1} f_h.
\end{equation}
Equivalently, it can be stated as:  find $U_h^n$ for $n=1,2,\ldots,N$ such that
\begin{equation}\label{eqn:fullysubd}
   \bPtau^\al U_h^n + A_hU_h^n=\bPtau^\al v_h+F_h^n,
   \quad \mbox{with} \quad U_h^0=v_h, \quad F_h^n=P_hf(t_n).
\end{equation}

In the same manner we derive a fully discrete scheme for the diffusion-wave
equation: find $U_h^n$ for $n=1,2,...,N$ such that
\begin{equation}\label{eqn:fullywaved}
   \bPtau^\al U_h^n + A_hU_h^n=\bPtau^\al v_h+(\bPtau^\al t)b_h+ F_h^n,\quad \mbox{with} \quad U_h^0=v_h, \quad F_h^n=P_hf(t_n).
\end{equation}
However, in view of the regularity result in Theorem \ref{thm:reg-inhomo-time}, we correct the last term
$F_h^n$ in the scheme \eqref{eqn:fullywaved} to $\bPtau\partial_t^{-1}f_h(t_n)$, in order to obtain better error estimates,
cf. Theorem \ref{thm:err-inhomog-fully-BE} and the remark afterwards. Hence we arrive at the following corrected scheme:
with $G_h^n=\partial_t^{-1}f_h(t_n)$, find $U_h^n$ for
$n=1,\ldots,N$ such that
\begin{equation}\label{eqn:fullywaved-mod}
  \bPtau^\al U_h^n + A_hU_h^n = \bPtau^\al v_h + (\bPtau^\al t)b_h + \bPtau G_h^n,\quad U_h^0 = v_h.
\end{equation}
In the BE scheme, the first-order convergence remains valid even if $g(0)\neq0$
\cite{Sanz-Serna:1988,LubichSloanThomee:1996}.

Next we turn to the SBD scheme. It is well known that the basic scheme \eqref{convol-quadrature}
is only first-order accurate if $g(0)\neq0$, e.g.,
for $g\equiv1$ \cite[Theorem 5.1]{Lubich:1988} \cite[Section 3]{CuestaLubichPlencia:2006}.
Hence, to get a second-order convergence one has to correct \eqref{convol-quadrature} properly.
We follow the approach proposed in \cite{LubichSloanThomee:1996,
CuestaLubichPlencia:2006}. Using the notation $\partial_t^{\beta} u$, $\beta<0$ for the Riemann-Liouville integral
$\partial_t^\beta u= \frac{1}{\Gamma(-\beta)}\int_0^t(t-s)^{-\beta-1}u(s)ds$ and the identity
\begin{equation*}
  (\partial_t^{\alpha}+A_h)^{-1}=\partial_t^{-\alpha} - (I+\partial^{-\alpha}_tA_h)^{-1}\partial_t^{-\alpha}A_h,
\end{equation*}
after splitting $f_h=f_{h,0}+\tilde{f}_h$, with $f_{h,0}=f_h(0)$ and $\tilde{f}_h=f_h-f_{h,0}$,
we rewrite the semidiscrete scheme \eqref{sol-semi} as
\begin{equation*}
  \begin{aligned}
    u_h(t)      &= v_h-(\partial_t^\al + A_h)^{-1} A_h v_h + (\partial_t^\al + A_h)^{-1}(f_{h,0}+\tilde{f}_h) \\
      &=   v_h-(\partial_t^\al + A_h)^{-1}\partial_t\partial_t^{-1} A_h v_h + (\partial_t^\al + A_h)^{-1}(\partial_t\partial_t^{-1}f_{h,0}+\tilde{f}_h).
  \end{aligned}
\end{equation*}
Now with $\bPtau^\al$ being convolution quadrature for the SBD formula we get
\begin{equation}\label{eqn:SBD0}
    U_h^n =   v_h-(\bPtau^\al + A_h)^{-1}\bPtau \partial_t^{-1} A_h v_h + (\bPtau^\al + A_h)^{-1}(\bPtau\partial_t^{-1}f_{h,0}+\tilde{f}_h).
\end{equation}
The purpose of keeping the operator $\partial_t^{-1}$ intact is to achieve a second-order accuracy.
Letting $1_\tau=(0,3/2,1,\ldots)$, using the identity $1_\tau=\bPtau\partial_t^{-1} 1$ at
grid points $t_n$ \cite{CuestaLubichPlencia:2006}, and the associativity \eqref{eqn:semig-disc}, the scheme \eqref{eqn:SBD0} can be rewritten as
\begin{equation*}
   (\bPtau^\al + A_h) (U_h^n-v_h) =  - 1_{\tau} A_h v_h + 1_\tau f_{h,0}+\tilde{f}_h.
\end{equation*}
Hence the second-order fully discrete scheme for the subdiffusion case reads:
with $F_h^n=P_hf(t_n)$ and $U_h^0=v_h$, find $U_h^n$, $n\ge 1$ such that
\begin{equation}\label{eqn:fullysub2nd}
\begin{split}
  \bPtau^\al U_h^1 + A_hU_h^1 + \tfrac12A_h U_h^0&=\bPtau^\al U_h^0+F_h^1+\tfrac12F_h^0,\\
  \bPtau^\al U_h^n  + A_hU_h^n &= \bPtau^\al U_h^0 + F_h^n, \quad n=2,\ldots,N.
\end{split}
\end{equation}
The modification at the first step maintains a second order convergence. Otherwise, due to the limited
smoothing property of the model \eqref{eqn:fde}, the scheme can only achieve a first-order convergence, unlike that
for the classical parabolic problem \cite{Thomee:2006}.

Similarly, for $1<\alpha<2$, we can derive a fully discrete scheme. In analogy with
the scheme \eqref{eqn:fullywaved-mod}, we correct the basic scheme in order to obtain error estimates
consistent with Theorem \ref{thm:reg-inhomo-time}. The corrected scheme reads:{
\begin{equation}\label{eqn:fullywave2nd-mod}
\begin{split}
 \bPtau^\al U_h^1 + A_h U_h^1 + \tfrac12A_h U_h^0&=\bPtau^\al U_h^0+ \bPtau^\al (tb_h) + \bPtau G_h^1+\tfrac12\bPtau G_h^0,\\
  \bPtau^\al U_h^n + A_hU_h^n &=\bPtau^\al U_h^0+ \bPtau^\al (tb_h)+ \bPtau G_h^n, \quad n=2,\ldots,N,
\end{split}
\end{equation}
with $U_h^0=v_h$ and $G_h^n=P_h\partial_t^{-1}f(t_n)$.

\begin{remark}\label{rmk:SBD}
It is known that without a correction the SBD scheme in general is only first-order accurate.
Lubich \cite{Lubich:1986,Lubich:1988} has developed various useful corrections to obtain second-order
accuracy. Even though these facts are well understood in the numerical PDEs community, it seems  not
so in the community of FDEs.
\end{remark}

%%%%%ssssssss
\subsection{Review of some existing methods}\label{ssec:existing}
Now we review several existing time stepping schemes in the subdiffusion case for the Caputo
derivative. The first is the popular L1 approximation of the fractional
derivative \cite{OldhamSpanier:1974,LinXu:2007}
\begin{equation*}
  \Dal \fy(t_n) \approx \frac{1}{\Gamma(1-\alpha)}\sum_{j=0}^{n-1}b_j\frac{\fy(t_{j+1})-\fy(t_j)}{\tau^\alpha}
\end{equation*}
with $b_j:=(j+1)^{1-\alpha}-j^{1-\alpha}$, $j=0,1,\ldots,n-1$. The local
truncation error is of the order $O(\tau^{2-\alpha})$, if the function
$\fy$ is twice continuously differentiable \cite{LinXu:2007}.

The second scheme, due to Zeng et al \cite{ZengLiLiuTurner:2013},
derived by applying convolution quadrature to a fractional integral
reformulation, is given by
\begin{equation*}
   D^{\alpha}\fy^n = L ^\alpha (\Delta \fy^n +f^n),
\end{equation*}
where the operators $D^\alpha$ and $L^\alpha$ are given by
\begin{equation*}
  \begin{aligned}
     D^\alpha \fy^n  = \tau^{-\alpha}\sum_{j=0}^n\omega_j(\fy^{n-j}-\fy^0)\quad
     \mbox{and}\quad     L^\alpha \fy^n = \frac{1}{2^\alpha} \sum_{j=0}^n\omega_j(-1)^j\fy^{n-k},
  \end{aligned}
\end{equation*}
with weights $\{\omega_j\}$ generated by the identity $(1-z)^\alpha = \sum_{j=0}^\infty \omega_j z^j$
\cite[formula (3.13)]{ZengLiLiuTurner:2013}. The third scheme is obtained in
the same spirit, but with $L^\alpha \fy^n = (1-\alpha/2)\fy^n + \alpha/2\fy^{n-1}$
\cite[formula (3.14)]{ZengLiLiuTurner:2013}.
Both the second and the third schemes converge at a rate $O(\tau^{2-\alpha})$, provided that $\fy$
is twice continuously differentiable, and at a rate $O(\tau^2)$, if further $\fy'(0)=0$.
These results were given in Table \ref{tab:existingscheme}.

\begin{remark}\label{rmk:Zeng}
The two schemes due to Zeng et al \cite{ZengLiLiuTurner:2013} are essentially a direct application of convolution
quadrature based on the trapezoidal rule and Newton-Gregory formula to the fractional integral term.
However, no correction for the first time step is incorporated, which will generally deteriorate the convergence.
\end{remark}

We have the following general comments on the schemes in Table \ref{tab:existingscheme}. The convergence
rates in Table \ref{tab:existingscheme} were mostly obtained by Taylor expansion, and thus requires
high solution regularity. The SBD scheme for the Riemann-Liouville derivative was recently analyzed in
\cite{LiDing:2014}, using Fourier transform, which uses substantially the zero extension $\bar \fy$
of $\fy$ for $t<0$. In particular, the assumption $\,{_{-\infty}^RD_t^{3-\alpha}} \bar \fy \in L^1(\mathbb{R})$
requires $\fy(0)=\fy'(0)=0$ and also $\fy''(0)=0$ for $\alpha$ close to zero. These conditions are
restrictive, and do not hold for homogeneous problems \cite{SakamotoYamamoto:2011}.

Finally, for the diffusion wave equation, one scheme
 is the Crank-Nicolson scheme \cite{SunWu:2006}. It approximates the Caputo
derivative $^C\kern-.5mm\partial_t^\alpha\fy(t_{n-1/2})$ by
\begin{equation*}
  ^C\kern-.5mm\partial_t^\alpha\fy(t_{n-1/2}) \approx  \frac{\tau^{-\al}}{\Gamma(3-\alpha)} \left [a_0\delta_t\fy^{n-1/2}
  - \sum_{j=1}^{n-1}(a_{n-j-1}-a_{n-j})\delta_t\fy^{j-1/2}-a_{n-1}\tau\fy'(0)\right],
\end{equation*}
where $\delta_t\fy^{j-1/2}=\fy^{j}-\fy^{j-1}$ denotes central difference, and
$a_j=(j+1)^{2-\alpha}-j^{2-\alpha}$. The local
truncation error is $O(\tau^{3-\alpha})$ for $C^3$ functions \cite{SunWu:2006}.

%%%%%%%%%%SSSSSSSSSSSSSSSSSSSS
\section{Error analysis}\label{sec:analysis}
%%%%%%%SSSSSSSSSS

Our goal in this section is to derive error estimates expressed directly in terms of problem data,
verifying the robustness of the proposed schemes. This is
done in two steps. First we bound the spatial error $u(t) - u_h(t)$, and then
the temporal error $ u_h(t_n) - U_h^n $. Throughout, let $\{(\la_j,\fy_j)\}_{j=1}^\infty$
be the Dirichlet eigenpairs of $-\Delta$ on $\Omega$, and $\{\fy_j\}_{j=1}^\infty$ form an orthonormal
basis in $L^2(\Omega)$. For any $q\ge0$, we denote by $\dH q\subset L^2\II$
the Hilbert space induced by the norm $ \|v\|_{\dH q}^2=\sum_{j=1}^\infty\la_j^q(v,\fy_j)^2$. Thus $\|v\|_{\dH 0}=\|v\|$ is the norm
in $L^2(\Om)$, $\|v\|_{\dH 1}$ the norm in $H_0^1(\Om)$ and $\|v\|_{\dH 2}=\|\Delta v\|$ is
equivalent to the norm in $H_0^1(\Omega)\cap H^2(\Om)$ \cite{Thomee:2006}.

%%%%%%%%%sssssss
\subsection{Error analysis of the semidiscrete scheme}\label{ssec:semidiscrete}
%%%%%%%%ssssssssssssss

The semidiscrete scheme \eqref{eqn:fdesemidis} for the subdiffusion case was already studied
\cite{JinLazarovZhou:2013,JinLazarovPasciakZhou:2013,JinLazarovPasciakZhou:2013a}. Hence we focus
on the diffusion wave case. We employ an operator technique developed in
\cite{FujitaSuzuki:1991} for the homogeneous problem, and an energy argument for the inhomogeneous
problem.

First we derive an integral representation of the solution $u$ to the homogeneous problem with
$f=0$ (see Appendix \ref{app:diffwave} for the solution theory). Since
$u:(0,T]\rightarrow L^2\II$ can be analytically extended to the sector $\{ z\in
\mathbb{C};z\neq0,|\arg z|<\pi/2 \}$ \cite[Theorem 2.3]{SakamotoYamamoto:2011}, we apply
the Laplace transform to \eqref{eqn:fde} to deduce
\begin{equation*}%\label{eqn:laptrans}
    z^\al \widehat u(z) + A\widehat u(z)=z^{\al-1} v + z^{\al-2} b,
\end{equation*}
with $A=-\Delta$ with a zero Dirichlet boundary condition. Hence the solution $u(t)$ can be represented by
\begin{equation}\label{eqn:interep}
    u(t)= \frac{1}{2\pi \mathrm{i}} \int_{\Gamma_{\theta,\delta}} e^{zt}(z^{\al}I+A)^{-1}(z^{\al-1} v + z^{\al-2} b)\,dz,
\end{equation}
where the contour $\Gamma_{\theta,\delta}$ is given by $
  \Gamma_{\theta,\delta}=\left\{z\in \mathbb{C}: |z|=\delta, |\arg z|\le \theta\right\}\cup
  \{z\in \mathbb{C}: z=\rho e^{\pm i\theta}, \rho\ge \delta\}.$
Throughout, we choose the angle $\theta$ such that $ \pi/2 < \theta < \min(\pi,\pi/\al)$ and hence
$z^{\al} \in \Sigma_{\theta'}$ with $ \theta'=\al\theta< \pi$ for all $z\in\Sigma_{\theta}:
=\{z\in\mathbb{C}: |\arg z|\leq \theta\}$. Then
there exists a constant $c$ which depends only on $\theta$ and $\al$ such that
\begin{equation}\label{eqn:resol}
  \| (z^{\al}I+A)^{-1} \| \le c|z|^{-\al},  \quad \forall z \in \Sigma_{\theta}.
\end{equation}
Similarly, with $A_h=-\Delta_h$, the solution $u_h$ to \eqref{eqn:fdesemidis} can be represented by
\begin{equation}\label{eqn:semi-interep}
    u_h(t)= \frac{1}{2\pi \mathrm{i}} \int_{\Gamma_{\theta,\delta}} e^{zt}(z^{\al}I+A_h)^{-1}(z^{\al-1} v_h + z^{\al-2} b_h)\,dz.
\end{equation}

The next lemma gives an important error estimate \cite{FujitaSuzuki:1991,BazhlekovaJinLazarovZhou:2014}.
\begin{lemma}\label{lem:wbound}
Let $ \fy\in L^2(\Omega) $,  $z\in \Sigma_{\theta}$,
 $w=(z^{\al}I+A)^{-1}\fy$, and $w_h=(z^{\al}I+A_h)^{-1}P_h \fy$. Then there holds
\begin{equation}\label{eqn:wboundHa}
    \|  w_h-w \|_{L^2(\Omega)} + h\| \nabla (w_h-w)\|_{L^2(\Omega)}
    \le Ch^2 \| \fy  \|_{L^2(\Omega)}.
\end{equation}
\end{lemma}

Now we state an error estimate for the homogeneous problem  \eqref{eqn:fdesemidis}.
\begin{theorem}\label{thm:err-homo-semi}
Let $1<\alpha<2$, and $u$ and $u_h$ be the solutions of problem \eqref{eqn:fde} with $v\in \dH q$,
$b\in \dH r$, $q,r\in [0,2]$, and $f=0$ and \eqref{eqn:fdesemidis}
with for $v_h=P_hv$, $b_h=P_hb$, $f_h=0$, respectively. The following estimate on $e_h(t)=u(t)-u_h(t)$ holds
\begin{equation*}
  \| e_h(t) \|_{L^2\II} + h\| \nabla e_h(t) \|_{L^2\II}  \le c h^{2} (t^{-\al({2-q})/2}\|v\|_{\dH q}+t^{1-\al({2-r})/2}\| b \|_{\dH r}).
\end{equation*}
\end{theorem}
\begin{proof} For $v,b\in L^2\II$,
by \eqref{eqn:interep} and \eqref{eqn:semi-interep}, $e_h(t)$ can be represented as
\begin{equation*}
e_h(t)=\frac1{2\pi\mathrm{i}}\int_{\Gamma_{\theta,\delta}} e^{zt}
\left(z^{\al-1}(w^v-w^v_h)+ z^{\al-2}(w^b-w^b_h)\right)\,dz,
\end{equation*}
with $w^v=(z^{\al}I+A)^{-1}v$, $w^b=(z^{\al}I+A)^{-1}b$,
$w^v_h=(z^{\al}I+A_h)^{-1}P_h v$ and $w^b_h=(z^{\al}I+A_h)^{-1}P_h b$.
By Lemma \ref{lem:wbound} and choosing $\delta=1/t$ in $\Gamma_{\theta,\delta}$ we have
\begin{equation*}
\begin{split}
\| \nabla e_h(t) \|_{L^2\II}
&\le c h\left(\int_{-\theta}^\theta e^{\cos\psi}t^{-\al} d\psi +\int_{1/t}^\infty e^{rt\cos\theta}
\rho^{\al-1}d\rho \right)\|v\|_{L^2\II} \\
&\quad + ch\left(\int_{-\theta}^\theta e^{\cos\psi}t^{1-\al}\,d\psi + \int_{1/t}^\infty e^{rt\cos\theta} \rho^{\al-2} \,d\rho \right)  \|b\|_{L^2\II} \\
&\le ch(t^{-\al} \| v \|_{L^2\II}+t^{1-\al} \| b \|_{L^2\II}).
\end{split}
\end{equation*}
A similar argument yields the $L^2$-estimate. This shows the assertion for $v,b\in L^2\II$.
Next for $v,b\in \dH 2$, first we consider the choice $v_h=R_hv$ and $b_h=R_hb$. Then $e_h(t)=u(t)-u_h(t)$ is given by
\begin{equation*}
  \begin{split}
    e_h(t)=&\frac1{2\pi\mathrm{i}}\int_{\Gamma_{\theta,\delta}}
    e^{zt} z^{\al-1}\left((z^{\al}I+A)^{-1}-(z^{\al}I+A_h)^{-1}R_h\right)v \,dz\\
    &+\frac1{2\pi\mathrm{i}}\int_{\Gamma_{\theta,\delta}}
    e^{zt} z^{\al-2}\left((z^{\al}I+A)^{-1}-(z^{\al}I+A_h)^{-1}R_h\right)b \,dz.
  \end{split}
\end{equation*}
Using the identity $z^{\al} (z^{\al}I+A)^{-1} = I-(z^{\al}I+A)^{-1}A$, we deduce
\begin{equation*}%\label{eqn:smooth-er-rep}
  \begin{split}
    e_h(t)=&\frac1{2\pi\mathrm{i}}\left(\int_{\Gamma_{\theta, 1/t}} e^{zt} z^{-1} (w^v(z)-w^v_h(z)) \,dz
    + \int_{\Gamma_{\theta, 1/t}} e^{zt} z^{-1} (v-R_hv) \,dz\right)\\
    &+\frac1{2\pi\mathrm{i}}\left(\int_{\Gamma_{\theta, 1/t}} e^{zt} z^{-2} (w^b(z)-w^b_h(z)) \,dz
    + \int_{\Gamma_{\theta, 1/t}} e^{zt} z^{-2} (b-R_hb) \,dz\right) := \mathrm{I}+\mathrm{II},
  \end{split}
\end{equation*}
where $w^v(z)=(z^{\al}I+A)^{-1}A v$ and $w^v_h(z)=(z^{\al}I+A_h)^{-1}A_hR_hv$, and $w^b(z)$ and $w^b_h(z)$
are defined analogously. Now Lemma \ref{lem:wbound} and the identity $A_hR_h=P_h A$ yield
\begin{equation*}
  \| w^v(t)-w_h^v(t) \|_{L^2\II} + h\| \nabla (w^v(t)-w_h^v(t)) \|_{L^2\II}
  \le ch^{2} \| Av \|_{L^2\II}.
\end{equation*}
Consequently, we can bound the first term $\mathrm{I}$ (with $\delta=1/t$)
\begin{equation*}
\begin{split}
  \|\mathrm{I}\|_{L^2\II} &\le c h^2 \| Av \|_{L^2\II}\bigg|\frac1{2\pi\mathrm{i}}\int_{\Gamma_{\theta,\delta}} e^{zt} z^{-1} \,dz\bigg|\\
      &\le  ch^2 \| Av \|_{L^2\II}\left(\int_{1/t}^\infty e^{rt\cos\theta} r^{-1} \,dr
      + \int_{-\theta}^{\theta} e^{\cos\psi} \,d\psi \right) \le c h^2 \| v \|_{\dH2}.
\end{split}
\end{equation*}
We derive a bound on the second term $\mathrm{II}$ in a similar way:
\begin{equation*}
  \|\mathrm{II}\|_{L^2\II} \le c h^2 \| Ab \|_{L^2\II}\bigg|\frac1{2\pi\mathrm{i}}\int_{\Gamma} e^{zt} z^{-2} \,dz\bigg|\leq
  c h^2 t\| b \|_{\dH2},
\end{equation*}
and the $L^2$-error estimate follows. The $H^1$-estimate is established analogously.
Last, for the choice $v_h=P_hv$ and $b_h=P_hb$, we have
\begin{equation*}
  E(t)v-E_hP_hv = E(t)v-E_hR_hv + E_h(R_hv-P_hv),
\end{equation*}
where $E$ and $E_h$ are continuous and semidiscrete solution operators, respectively (see Appendix \ref{app:diffwave} for the definitions).
The first term is already bounded. By Theorem \ref{lem:discrete-reg} in Appendix \ref{app:diffwave} and approximation properties
of $P_h$ and $R_h$, there holds
\begin{equation*}
\| E_h(t)(P_hv-R_hv) \|_{\dH p} \le c \| P_hv-R_hv \|_{\dH p} \le ch^{2-p}\|v\|_{\dH2}, \ \ p =0,1.
\end{equation*}
The estimate for $b\in \dH 2$ follows analogously.
These estimates and interpolation complete the proof of the theorem.
\end{proof}

For problem \eqref{eqn:fde} with $f\in L^\infty(0,T;L^2\II)$, we have the following result.
The proof is identical to \cite[Theorem 3.2]{JinLazarovPasciakZhou:2013a}, and
hence omitted. The factor $\ell_h^2$ reflects the limited smoothing property in space
of the diffusion wave operator, cf. Theorem \ref{thm:reg-inhomo-space}.
\begin{theorem}\label{thm:err-inhomo-semi}
Let $1<\alpha<2$, $u$ and $u_h$ be the solutions
of \eqref{eqn:fde} with $v,b=0$, $f\in L^\infty(0,T;L^2\II)$, and \eqref{eqn:fem} with $v_h=b_h=0$, $f_h=P_hf$, respectively.
Then with $\ell_h =| \ln h|$, there holds for $e_h(t)=u(t)-u_h(t)$
\begin{equation*}
 \| e_h(t) \|_{L^2(\Om)} +
h  \|\nabla e_h(t)\|_{L^2(\Omega)} \le ch^{2} \ell_h^{2} \|f\|_{L^\infty(0,t; L^2\II)}.
\end{equation*}
\end{theorem}

%%%%%%%%%%%%%%%%%%%%%%%%%%%%%%%%%%%%%%%%%%%%%
\subsection{Error analysis for BE method}
Now we derive $L^2$ error estimates for the fully discrete schemes \eqref{eqn:fullysubd} and
\eqref{eqn:fullywaved-mod} using the framework developed in \cite{Lubich:1988,CuestaLubichPlencia:2006}.
Alternatively, one can analyze the schemes by directly bounding the kernel function in the resolvent
\cite{LubichSloanThomee:1996,JinLazarovZhou:2015ima}.
We begin with an important result \cite[Theorem 5.2]{Lubich:1988}.
\begin{lemma}\label{lem:boundBE}
Let $K(z)$ be analytic in $\Sigma_{\theta}$ and \eqref{eqn:Kbound} hold.
Then for $g(t)=ct^{\beta-1}$, the convolution quadrature based on the
BE method satisfies
\begin{equation*}
    \| (K(\partial_t)-K(\bPtau))g(t)  \| \le
    \left\{ \begin{array}{ll}
     ct^{\mu-1}\tau^\beta, &0<\beta\le 1, \\
     ct^{\mu+\beta-2}\tau, &\beta\ge1.\end{array}\right.
\end{equation*}
\end{lemma}

First we state an error estimate for the homogeneous subdiffusion problem.
\begin{theorem}\label{thm:err-homo-fully-BE}
Let $u$ and $U_h^n$ be the solutions of problem \eqref{eqn:fde} with $v\in \dH q$, $q\in[0,2]$, and $f=0$ and \eqref{eqn:fullysubd}/\eqref{eqn:fullywaved-mod}
with $v_h=P_hv$ and $f_h=0$, respectively. Then the following statements hold.
\begin{itemize}
\item[(i)] If $0<\alpha<1$, then
%Let $0<\alpha<1$, and $u$ and $U_h^n$ be the solutions of problem \eqref{eqn:fde} with $v\in \dH q$,
%$q\in[0,2]$, and $f=0$, and \eqref{eqn:fullysubd} with $v_h=P_hv$, respectively. Then
\begin{equation*}
   \| u(t_n)-U_h^n \|_{L^2(\Om)} \le c(t_n^{q\al/2-1}\tau + t_n^{(q-2)\alpha/2}h^2)  \| v \|_{\dH q}.
\end{equation*}
\item[(ii)] If $1<\alpha<2$, $b\in \dH r$, $r\in[0,2]$, and $b_h=P_hb$, then
\begin{equation*}
  \begin{aligned}
   \| u(t_n)-U_h^n \|_{L^2(\Om)} & \le c(t_n^{q\alpha/2-1}\tau + t_n^{(q-2)\alpha/2}h^{2}) \|v\|_{\dH q}\\
    & \quad + c(t_n^{r\al/2}\tau + t_n^{1+(r-2)\alpha/2}h^2) \|b\|_{\dH r}.
  \end{aligned}
\end{equation*}
\end{itemize}
\end{theorem}
\begin{proof}
In view of the semidiscrete error estimates \cite[Section 3]{JinLazarovZhou:2013}
(where the log factor $\ell_h$ in the estimates in \cite[Section 3]{JinLazarovZhou:2013}
can be removed using the operator trick in Section \ref{ssec:semidiscrete}), it suffices
to bound $U_h^n-u_h(t_n)$. To this end, we denote for $z\in \Sigma_\theta$,
$\theta\in(\pi/2,\pi)$, $G(z)= z^\al (z^\al I +  A_h)^{-1}$. Then by \eqref{sol-semi} and \eqref{sol-be},
we have
\begin{equation}\label{eqn:diff1}
 U_h^n-u_h(t_n)= (G(\bPtau)-G(\partial_t))v_h.
\end{equation}
By \eqref{eqn:resol}, there holds $G(z)\le c$ for $z\in \Sigma_\theta$. Hence, for $v\in L^2\II$,
\eqref{eqn:diff1}, Lemma \ref{lem:boundBE} (with $\mu=0$ and $\beta=1$), and the $L^2\II$ stability
of $P_h$ give
\begin{equation}\label{eqn:err-1st-subdiff}
 \| u_h(t_n)-U_h^n \|_{L^2(\Om)} \le c \tau t_n^{-1} \|v_h\|_{L^2\II}\leq ct_n^{-1}\tau\|v\|_{L^2\II}.
\end{equation}
For $v\in \dH 2$, first consider the choice $v_h=R_hv$.
Using the identity $G(z)=I-(z^\al I+A_h)^{-1}A_h$, with $G_s(z)=(z^\al I+A_h)^{-1}$, we have
$U_h^n-u_h(t_n)= (G_s(\bPtau)-G_s(\partial_t))A_hv_h.$
Appealing to \eqref{eqn:resol} and Lemma \ref{lem:boundBE} (with $\mu=\al$ and $\beta=1$) gives
\begin{equation*}
     \| u_h(t_n)-U_h^n \|_{L^2(\Om)} \le c \tau t_n^{\al-1} \|A_h v_h\|_{L^2\II}\leq c\tau t_n^{\al-1}\|v\|_{\dH 2},
\end{equation*}
where the last line follows from $A_hR_h=P_hA$. The estimate holds also for the choice $v_h=P_hv$
in view of the $L^2\II$ stability of the scheme, cf. \eqref{eqn:err-1st-subdiff}, and the argument
in the proof of Theorem \ref{thm:err-homo-semi}. The assertion
now follows from interpolation. The case of $1<\alpha<2$ is analogous, and hence the proof is omitted.
\end{proof}
%
%The case of the homogeneous diffusion wave problem is analogous.
%\begin{theorem}\label{thm:err-homo-fully-BE-wave}
%Let $1<\alpha<2$, and $u$ and $U_h^n$ be the solutions of problem \eqref{eqn:fde} with $v\in \dH q$,
%$b\in \dH r$, $q,r\in[0,2]$, $f=0$ and \eqref{eqn:fullywaved} with
%$v_h=P_hv$, $b_h=P_hb $ and $f_h=0$, respectively. Then there holds
%\begin{equation*}
%   \| u(t_n)-U_h^n \|_{L^2(\Om)} \le c(t_n^{q\alpha/2-1}\tau + t_n^{(q-2)\alpha/2}h^{2}) \|v\|_{\dH q}
%    + c(t_n^{r\al/2}\tau + t_n^{1+(r-2)\alpha/2}h^2) \|b\|_{\dH r}.
%\end{equation*}
%\end{theorem}

Last we give error estimates for the BE method for problem \eqref{eqn:fde} with $f\neq 0$
but $v=0$ (also $b=0$, if $1<\alpha<2$).
\begin{theorem}\label{thm:err-inhomog-fully-BE}
Let $u$ be the solution of problem \eqref{eqn:fde} with homogeneous initial data and $f\in
L^\infty(0,T;L^2\II)$, and $U_h^n$ be the solution to \eqref{eqn:fullysubd}/\eqref{eqn:fullywaved-mod} with $f_h=P_hf$.
Then with $\ell_h=|\ln h|$, the following statements hold:
\begin{itemize}
  \item[(i)] For $0<\alpha<1$, % let $U_h^n$ be the solution to \eqref{eqn:fullysubd}.
  if   $\int_0^t (t-s)^{\al-1}  \| f'(s)  \|_{L^2\II}ds<\infty$ for $t\in(0,T],  $
  then
\begin{equation*}
  \begin{aligned}
    \| u(t_n)-U_h^n \|_{L^2\II} & \le
    c  (h^2\ell_h^2\| f \|_{L^\infty(0,T;L^2\II)}+ t_n^{\al-1}\tau \|  f(0) \|_{L^2\II} \\
       &  \qquad +       \tau \int_0^{t_n} (t_n-s)^{\al-1}\| f'(s)\|_{L^2\II} \,ds ) .
  \end{aligned}
\end{equation*}
\item[(ii)] For $1<\alpha<2$, %let $U_h^n$ be the solution to \eqref{eqn:fullywaved-mod}.
then
\begin{equation*}
 \| u(t_n)-U_h^n  \|_{L^2\II}\le c(h^2\ell_h^2+\tau ) \| f  \|_{L^\infty(0,T;L^2\II)}.
\end{equation*}
\end{itemize}
\end{theorem}
\begin{proof}
With $G(z)=(z^\al I+A_h)^{-1}$, the semidiscrete solution $u_h$ and fully discrete solution $U_h^n$
are given by  $ u_h = G(\partial_t)f_h$ and $U_h^n= G(\partial_\tau)f_h$,  respectively.
Using the splitting $ f_h(t) = f_h(0)+(1\ast f_h')(t)$ and the convolution relation \eqref{eqn:semig-disc}, we have
\begin{equation*}
\begin{split}
   u_h(t_n) - U_h^n &= \left( G(\partial_t)-G(\partial_\tau)\right) (f_h(0)+(1\ast f_h')(t_n)) \\
  &=\left( G(\partial_t)-G(\partial_\tau)\right) f_h(0)+\left( (G(\partial_t)-G(\partial_\tau))1\right)\ast f_h'(t_n)):=\mathrm{I}+\mathrm{II}.
\end{split}
\end{equation*}
Then Lemma \ref{lem:boundBE} (with $\mu=\al$ and $\beta=1$) yields a bound on the first term $\mathrm{I}$
\begin{equation*}
  \| \mathrm{I}\|_{L^2\II}\le c\tau t_n^{\al-1}\| f_h(0)  \|_{L^2\II}\le c\tau t_n^{\al-1}\| f(0)  \|_{L^2\II}.
\end{equation*}
Likewise, the term $\mathrm{II}$ can be bounded using Lemma \ref{lem:boundBE} and the $L^2$ stability of $P_h$ by
\begin{equation*}
    \begin{split}
   \| \mathrm{II}\|_{L^2\II}
   &\le \int_0^{t_n} {\|\left( (G(\partial_t)-G(\partial_\tau))1\right)(t_n-s) f_h'(s)\|_{L^2\II}} \,ds  \\
   &\le c\tau \int_0^{t_n} (t_n-s)^{\al-1}\| f_h'(s)\|_{L^2\II} \,ds
   \le c\tau \int_0^{t_n} (t_n-s)^{\al-1}\| f'(s)\|_{L^2\II} \,ds.
    \end{split}
\end{equation*}
 This shows assertion (i).
For the scheme \eqref{eqn:fullywaved-mod} with $ v =b=0$, $U_h^n=G(z)g_h$ with $g_h=\partial_t^{-1} f_h$ and $G(z)=z(z^\al I+A_h)^{-1}$.
Hence the equality $g_h=1\ast f_h$, the convolution rule \eqref{eqn:semig-disc} and Lemma \ref{lem:boundBE} with $\mu=\al-1$ and $\beta=1$ yield
\begin{equation*}
 \| u_h(t_n)-U_h^n  \|_{L^2\II} \le c\tau\int_0^{t_n} (t_n-s)^{\al-2} \|  f(s) \|_{L^2\II} \,ds\leq c_T\tau\|f\|_{L^\infty(0,T;L^2\II)},
\end{equation*}
from which follows directly Assertion (ii), and this completes the proof.
\end{proof}

\begin{remark}\label{rem:mod}
Assertion (i) in Theorem \ref{thm:err-inhomog-fully-BE} holds also for the basic scheme \eqref{eqn:fullywaved}.
However, the corrected scheme \eqref{eqn:fullywaved-mod} is uniformly first order in time for $v=b=0$ and
$f\in L^\infty(0,T;L^2\II)$, which is consistent with the temporal regularity
result in Theorem \ref{thm:reg-inhomo-time}.
Hence, the correction in \eqref{eqn:fullywaved-mod} gives better error estimates.
\end{remark}

\subsection{Error analysis for the SBD scheme}
Now we turn to the analysis of the SBD scheme. Like Lemma \ref{lem:boundBE},
the following estimate holds \cite[Theorem 5.2]{Lubich:1988}.
\begin{lemma}\label{lem:boundSBD}
Let $K(z)$ be analytic in $\Sigma_{\theta}$ and \eqref{eqn:Kbound} hold. Then for
$g(t)=ct^{\beta-1}$, the convolution quadrature based on the SBD scheme satisfies
\begin{equation*}
    \| (K(\partial_t)-K(\bPtau))g(t)  \| \le \left\{ \begin{array}{ll}
     ct^{\mu-1}\tau^\beta, &~~0<\beta\le 2, \\
     ct^{\mu+\beta-3}\tau^2, &~~\beta\ge2.\end{array}\right.
\end{equation*}
\end{lemma}

Now we state the following error estimates for the homogeneous problem.
\begin{theorem}\label{thm:err-homo-fully-SBD}
Let $u$ and $U_h^n$ be the solutions of problem \eqref{eqn:fde} with  $v\in \dH q$, $q\in[0,2]$, and $f=0$ and \eqref{eqn:fullysub2nd}/\eqref{eqn:fullywave2nd-mod} with
$v_h= P_h v$ and $f_h=0$, respectively. Then the following statements hold.
\begin{itemize}
  \item[(i)] If $0<\al<1$, then
  \begin{equation*}
   \| u(t_n)-U_h^n \|_{L^2(\Om)} \le c(\tau^2 t_n^{-2+{q\al}/{2}} +    h^{2}t_n^{-{(2-q)\al}/{2}}) \|v\|_{\dH q}.
  \end{equation*}
  \item[(ii)] If $1<\al<2$, $b\in \dH r$, $r\in[0,2]$, and $b_h=P_hb$, then
\begin{equation*}
\begin{split}
   \| u(t_n)-U_h^n \|_{L^2(\Om)} \le & c(\tau^2t_n^{-2 +{q\al}/{2}}
   +h^{2} t_n^{-({2-q})\al/{2}}) \|v\|_{\dH q} \\
   & \qquad+ c(\tau^2 t_n^{{r\al}/{2}-1} +h^{2} t_n^{1-({2-r})\al/{2}}) \|b\|_{\dH r}.
   \end{split}
\end{equation*}
\end{itemize}
\end{theorem}
\begin{proof}
We provide the proof only for $0<\alpha<1$, since that for $1<\alpha<2$ is identical.
For $v\in L^2\II$, the difference between $u_h(t_n)$ and $U_h^n$ is given by
\begin{equation*}
 u_h(t_n)-U_h^n = (G(\partial_t^\al)-G(\bPtau)) \partial_t^{-1}(A_h v_h)(t_n),
\end{equation*}
where $G(z)=-z(z^\al I+A_h)^{-1}A_h$.
By \eqref{eqn:resol} and the identity
\begin{equation*}
   G(z) = -z(z^\al I+A_h)^{-1}A_h = -z I+z^{\al+1}(z^\al I+A_h)^{-1}) \quad \forall z\in \Sigma_{\pi-\theta},
\end{equation*}
there holds $\| G(z)\| \le c|z|$, for $z\in \Sigma_{\theta}.$ Then Lemma \ref{lem:boundSBD} (with
$\mu=-1$ and $\beta=2$) and the $L^2\II$ stability of $P_h$ give
\begin{equation}\label{eqn:err-2nd-subdiff}
  \| U_h^n-u_h(t_n) \|_{L^2(\Omega)} \le c\tau^2 t_n^{-2} \| v_h \|_{L^2(\Omega)}\leq c\tau^2t_n^{-2}\|v\|_{L^2\II}.
\end{equation}
For smooth initial data $v\in \dH 2$, we first set $U_h^0=v_h=R_hv$.
By setting $G_s(z)= -z(z^\al I+A_h)^{-1}$, the difference $U_h^n-u_h(t_n)$ can be written by
\begin{equation*}%\label{eqn:diff2}
 U_h^n-u_h(t_n)= (G_s(\bPtau)-G_s(\partial_t))A_h v_h.
\end{equation*}
From \eqref{eqn:resol}, we deduce $\| G_s(z)\| \le  M |z|^{1-\al}$ for all $z\in \Sigma_{\theta}.$
Now Lemma \ref{lem:boundSBD} (with $\mu=\al-1$ and $\beta=2$) and the identity $A_hR_h=P_hA$ gives
\begin{equation*}
  \|U_h^n-u_h(t_n) \|_{L^2(\Omega)} \le
  c\tau^2 t_n^{\al-2} \| A_h v_h \|_{L^2(\Omega)} \leq c\tau^2t_n^{\alpha-2}\|v\|_{\dH 2}.
\end{equation*}
Last, the desired assertion follows from the $L^2$ stability of the scheme \eqref{eqn:fullysub2nd}
(as a direct consequence of \eqref{eqn:err-2nd-subdiff}, cf. the proof of
Theorem \ref{thm:err-homo-fully-BE}) and interpolation.
\end{proof}

Last we give error estimates for the inhomogeneous problem.
\begin{theorem}\label{thm:inhomog-fully-SBD}
Let $u$ be the solution of problem \eqref{eqn:fde} with homogeneous initial data and $f\in L^\infty(0,T;
L^2\II)$, and $U_h^n$ be the solution to \eqref{eqn:fullysub2nd}/\eqref{eqn:fullywave2nd-mod} with $f_h=
P_hf$. Then with $\ell_h=|\ln h|$, the following statements hold.
\begin{itemize}
  \item[(i)] For $0<\alpha<1$, if $\int_0^t (t-s)^{\al-1}  \| f''(s)  \|_{L^2\II}ds<\infty$ for $t\in[0,T]$, then
\begin{equation*}
\begin{split}
\| u(t_n)-U_h^n &\|_{L^2\II}\le c\Large(h^2\ell_h^2\| f \|_{L^\infty(0,T;L^2\II)}+ t_n^{\al-2}\tau^2 \|  f(0) \|_{L^2\II}\\
&+t_n^{\al-1}\tau^2 \|  f'(0) \|_{L^2\II}+
\tau^2 \int_0^{t_n} (t_n-s)^{\al-1}\| f''(s)\|_{L^2\II} \,ds \Large).
\end{split}
\end{equation*}
\item[(ii)] For $1<\alpha<2$, if $\int_0^t (t-s)^{\al-2}  \| f'(s)  \|_{L^2\II}ds<\infty$ for $t\in[0,T]$, then
\begin{equation*}
\begin{split}
 \| u(t_n)-U_h^n & \|_{L^2\II} \le ch^2\ell_h^2\| f \|_{L^\infty(0,T;L^2\II)}  \\
 &+c\tau^2(t_n^{\al-2}\|  f(0) \|_{L^2\II}+\int_0^{t_n} (t_n-s)^{\al-2} \|  f'(s) \|_{L^2\II} \,ds).
 \end{split}
\end{equation*}
\end{itemize}
\end{theorem}
\begin{proof}
By \cite[Theorem 3.2]{JinLazarovPasciakZhou:2013a} and Theorem \ref{thm:err-inhomo-semi}, it suffices to bound $e_h^n=u_h(t_n)-U_h^n$.
Upon letting $G_1(z)=z(z^\al I+A_h)^{-1}$ and $G_2(z)=(z^\al I+A_h)^{-1}$ and using the identity
$\tilde f_h = f_h(0) + tf_h'+t\ast f_h''$, the solutions $u_h(t_n)$ and $U_h^n$ are represented by
\begin{equation*}
  \begin{aligned}
  u_h(t_n) &= G_1(\partial_t)tf_h(0) + G_2(\partial_t)tf_h'(0) + (G_2(\partial_t)t)\ast f_h'', \quad \mbox{and}\\
  U_h^n &= G_1(\partial_\tau)tf_h(0) + G_2(\partial_\tau)tf_h'(0) + (G_2(\partial_\tau)t)\ast f_h'',
  \end{aligned}
\end{equation*}
respectively. Therefore,
\begin{equation*}
\begin{split}
u_h(t_n)-U_h^n&=\left(G_1(\partial_t)-G_1(\partial_\tau)\right)tf_h(0)+\left(G_2(\partial_t)-G_2(\partial_\tau)\right)tf_h'(0)\\
&\quad+\left((G_2(\partial_t)-G_2(\partial_\tau))t\right)\ast f_h'':=\mathrm{I}+\mathrm{II}.
\end{split}
\end{equation*}
Lemma \ref{lem:boundSBD} (with $\mu=1-\al$ and $\beta=2$) gives a bound on the first term $\mathrm{I}$
as $\|\mathrm{I}\|_{L^2\II}\le  ct_n^{\al-2}\tau^2.$ The bound on the second term $\mathrm{II}$ follows
from Lemma \ref{lem:boundSBD} (with $\mu=\al$ and $\beta=2$) and the $L^2\II$ stability of $P_h$ by
\begin{equation*}
\begin{split}
\|\mathrm{II}\|_{L^2\II} &\le
\|  \left(G_2(\partial_t)-G_2(\partial_\tau)\right)tf_h'(0) \|_{L^2\II}\\
&\quad + \int_0^{t_n}  \|  \left(G_2(\partial_t)-G_2(\partial_\tau)t\right)(t_n-s)f_h''(s) \|_{L^2\II} \,ds\\
&\le ct_n^{\al-1}\tau^2 \|  f'(0) \|_{L^2\II}+
c\tau^2 \int_0^{t_n} (t_n-s)^{\al-1}\| f''(s)\|_{L^2\II} \,ds.
\end{split}
\end{equation*}
This shows assertion (i). Assertion (ii) follows analogously from Lemma \ref{lem:boundBE} with $\mu=\al-1$ and $\beta=2$.
\end{proof}

\begin{remark}\label{rem:uniform}
For subdiffusion with $v=0$, the SBD scheme \eqref{eqn:fullysub2nd} is of uniformly second order if $f(0)=f'(0)=0$
and $\int_0^t (t-s)^{\al-1}  \| f''(s)  \|_{L^2\II}ds<\infty$ for $t\in[0,T]$, and for diffusion-wave, the
SBD scheme \eqref{eqn:fullywave2nd-mod} is of uniformly second order if $f(0)=0$ and $\int_0^t (t-s)^{\al-1}
\| f'(s)  \|_{L^2\II}ds<\infty$ for $t\in[0,T]$. These conditions are directly verifiable. It is noteworthy
that the estimate in Theorem \ref{thm:inhomog-fully-SBD}(i) holds also for an uncorrected SBD scheme for
the diffusion wave problem.
\end{remark}

\section{Numerical experiments and discussions}\label{sec:numeric}
Now we illustrate the convergence and robustness of the schemes on several examples on the square domain
$\Omega=(0,1)^2$. For the examples below, the exact solution can be expressed as an infinite series involving
the Mittag-Leffler function $E_{\al,\beta}(z)$ \cite{SakamotoYamamoto:2011,JinLazarovZhou:2013}, for which we
employ an algorithm developed in \cite{Seybold:2008}. In the computations, the domain $\Omega=(0,1)^2$ is
triangulated as follows. It is first divided into $M^2$ small equal squares, by partitioning the unit interval
$(0,1)$ into $M$ equally spaced subintervals, and the diagonal of each small square is then connected to obtain
a symmetric triangulation. Likewise, we fix the time step size $\tau$ at $\tau=t/N$, where $t$ is the time of
interest. To examine the spatial and temporal convergence rates separately, we take a small time step size $\tau$
(or mesh size $h$, respectively), so that the temporal (or spatial) discretization error is negligible. We measure
the error $e^n =: u(t_n)-U_h^n$ by the normalized error $\| e^n \|_{L^2\II}/\| v \|_{L^2\II}$ and (also $\| e^n
\|_{\dot H^1\II}/\| v \|_{L^2\II}$ for spatial convergence).

\subsection{Subdiffusion}
We consider the following three examples:
\begin{itemize}
  \item[(a)]  $v=xy(1-x)(1-y)\in\dH 2$ and $f=0$;
  \item[(b)] $v=\chi_{(0,1/2]\times(0,1)}(x,y) \in \dH {{1/2}-\epsilon}$ with $\ep\in(0,1/2)$ and $f=0$.
  \item[(c)] $v=0$, and $f=(1+t^{0.2})\chi_{(0,1/2]\times(0,1)}(x,y)$.
\end{itemize}

Since the spatial convergence of the semidiscrete Galerkin scheme for subdiffusion was studied in
\cite{JinLazarovZhou:2013,JinLazarovPasciakZhou:2013}, we focus on the temporal convergence at $t=0.1$.
The numerical results for cases (a) and (b) are shown in Tables  \ref{tab:subdiff-time-a} and
\ref{tab:subdiff-time-b}, respectively, where \texttt{rate} refers to the empirical convergence
rate of the errors, and the numbers in the bracket
denote theoretical predictions. The BE and SBD schemes exhibit a very steady first and
second-order convergence, respectively, for both smooth and nonsmooth data, which
agree well with the convergence theory.

In Tables \ref{tab:subdiff-time-a} and \ref{tab:subdiff-time-b}, we also include results by four
existing schemes: the L1 scheme \cite{LinXu:2007} (denoted by Lin-Xu), two schemes from Zeng
et al \cite{ZengLiLiuTurner:2013} (Zeng I and Zeng II), where the theoretical rates in the bracket
are for smooth solutions, all at a rate $O(\tau^{2-\alpha})$, cf. Table \ref{tab:existingscheme}.
%and the discontinuous Petrov-Galerkin scheme (denoted by DPG) \cite{Mustapha:2014DG} of Mustapha et al.
For both cases, the L1 scheme only achieves a first-order convergence. This is not accidental:
its best possible convergence rate for homogeneous problems is $O(\tau)$ \cite{JinLazarovZhou:2015ima}.
The convergence of the Zeng I scheme strongly depends on  $\alpha$, and can fail to achieve an
$O(\tau)$ rate for nonsmooth data. Their second scheme can only achieve a first-order
convergence in either case. Hence, time stepping schemes derived under the assumption that the 
solution is smooth may be not robust with respect to data regularity, which motivates revisiting 
their analysis for nonsmooth data. In contrast, the schemes proposed in this work are robust.
Further, we include the numerical results for the discontinuous Petrov-Galerkin scheme
(denoted by DPG) of Mustapha et al \cite{Mustapha:2014DG}, using piecewise linear functions.
The DPG scheme was analyzed for graded meshes in \cite{Mustapha:2014DG}, which compensates the 
solution singularity with local refinement, and the optimal convergence rates for uniform meshes 
remain unknown. Numerically, we observe that it merits second-order
convergence for case (a), but for nonsmooth case (b), the convergence seems not so steady.

\begin{table}[htb!]
\caption{The $L^2$-error for case (a) for
 $t=0.1$, $\al=0.5$, and $h=1/512$.}
\label{tab:subdiff-time-a}
\begin{center}
\vspace{-.3cm}
     \begin{tabular}{|c|c|cccccc|c|}
     \hline
      $\alpha\backslash N$ & method &$10$ &$20$ & $40$ & $80$  &$160$ & $320$ &rate \\
     \hline
               & BE   &2.96e-4 &1.46e-4 &7.27e-5 &3.63e-5 &1.81e-5  &9.05e-6  & 1.00 (1.00) \\
     %\cline{2-9}
               & SBD   &2.94e-5 &6.88e-6 &1.66e-6 &4.09e-7 &1.01e-7  &2.49e-8 & 2.02 (2.00) \\
     %\cline{2-9}
     $0.1$ &Lin-Xu    &2.76e-4 &1.35e-4 &6.71e-5 &3.34e-5 &1.66e-5 &8.31e-6  & 1.02 (1.90) \\
    % \cline{2-9}
               & Zeng I  &1.12e-2 &5.84e-3 &3.03e-3 &1.57e-3 &8.05e-4  &4.12e-4 & 0.96 (1.90)\\
     %\cline{2-9}
               & Zeng II  &2.40e-4 &1.20e-4 &5.98e-5 &2.99e-5 &1.49e-5 &7.47e-6 & 1.00 (1.90)\\
        &  { DPG} &7.99e-1 &6.77e-1 &4.79e-1 &2.33e-1 &5.64e-2 &1.52e-2 & 1.88 ($--$)\\
     \hline
                    & BE    &3.33e-3 &1.63e-3 &8.05e-4 &4.00e-4  &1.99e-4 &9.96e-5 & 1.00 (1.00) \\
     %\cline{2-9}
               & SBD    &4.18e-4 &9.70e-5 &2.33e-5 &5.71e-6 &1.48e-6 &3.45e-7 & 2.02 (2.00) \\
     %\cline{2-9}
     $0.5$ &Lin-Xu     &2.45e-3 &1.17e-3 &5.68e-4 &2.80e-4 &1.38e-4 &6.88e-5 & 1.01 (1.50)\\
     %\cline{2-9}
               & Zeng I  &1.25e-2 &3.30e-3 &9.01e-4 &2.69e-4 &9.14e-5 &3.58e-5 & 1.85 (1.50)\\
    % \cline{2-9}
               & Zeng II  &9.68e-4 &5.15e-4 &2.65e-4 &1.35e-4 &6.75e-5 &3.39e-5 & 0.95 (1.50)\\
 &   {    DPG} &1.58e-2 &3.64e-3 &9.35e-4 &2.48e-4 &6.64e-5 &1.74e-5 & 1.93 ($--$)\\

     \hline
                    & BE    &1.89e-2 &9.42e-3 &4.70e-3 &2.35e-3 &1.17e-3 &5.85e-4 & 1.00 (1.00) \\
     %\cline{2-9}
               & SBD    &2.53e-3 &5.98e-4 &1.45e-4 &3.59e-5 &8.88e-6  &2.19e-6 & 2.02 (2.00) \\
     %\cline{2-9}
     $0.9$ &Lin-Xu    &1.78e-2 &8.51e-3 &4.08e-3 &1.97e-3 &9.50e-4 &4.61e-4 & 1.04 (1.10)\\
    % \cline{2-9}
               & Zeng I   &3.85e-4 &1.45e-4 &1.10e-4 &6.59e-5 &3.59e-5 &1.87e-5 & 0.91 (1.10)\\
     %\cline{2-9}
               & Zeng II  &6.07e-4 &2.27e-4 &1.65e-4 &9.55e-5 &5.10e-5 &2.63e-5 & 0.93 (1.10)\\
        &  {  DPG} &9.03e-4 &2.13e-4 &5.08e-5 &1.21e-5 &2.84e-6 &5.96e-7 & 2.06 ($--$)\\
     \hline
     \end{tabular}
\end{center}
\end{table}

\begin{table}[htb!]
\caption{The $L^2$-error for case (b) for
 $t=0.1$, $\al=0.5$, and $h=1/512$.}
\label{tab:subdiff-time-b}
\begin{center}
\vspace{-.3cm}
     \begin{tabular}{|c|c|cccccc|c|}
     \hline
      $\alpha\backslash N$ & method  &$10$ &$20$ & $40$ & $80$ & $160$ & $320$ &rate \\
     \hline
               & BE     &1.86e-4 &9.21e-5 &4.58e-5 &2.28e-5 &1.14e-5  &5.69e-6  & 1.00 (1.00) \\
     %\cline{2-9}
               & SBD    &1.85e-5 &4.34e-6 &1.05e-6 &2.58e-7 &6.38e-8 &1.57e-8 & 2.02 (2.00) \\
     %\cline{2-9}
     $0.1$ &Lin-Xu      &1.74e-4 &8.51e-5 &4.22e-5 &2.10e-5 &1.05e-5 &5.23e-6  & 1.00 (1.90)\\
     %\cline{2-9}
               & Zeng I  &1.21e-2 &6.42e-3 &3.39e-3 &1.80e-3  &9.42e-4 &4.95e-4 & 0.93 (1.90)\\
     %\cline{2-9}
               & Zeng II  &1.51e-4 &7.53e-5 &3.76e-5 &1.88e-5 &9.40e-6 &4.70e-6 & 1.00 (1.90)\\
      &  {   DPG} &8.98e-1 &8.38e-1 &7.41e-1 &6.14e-1 &4.93e-1 &4.02e-1 & 0.29 ($--$)\\
     \hline
                & BE      &2.09e-3 &1.02e-3 &5.05e-4 &2.51e-4 &1.25e-4 &6.25e-5 & 1.01 (1.00) \\
     %\cline{2-9}
                &SBD      &2.64e-4 &6.11e-5 &1.47e-5 &3.60e-6 &8.87e-7 &2.17e-7  & 2.02 (2.00) \\
     %\cline{2-9}
     $0.5$     &Lin-Xu    &1.52e-3 &7.26e-4 &3.54e-4 &1.74e-4 &8.64e-5 &4.30e-5 & 1.00 (1.50)\\
     %\cline{2-9}
               & Zeng I   &7.87e-2 &4.61e-2 &2.70e-2 &1.56e-2 &8.95e-3 &5.05e-3 & 0.80 (1.50)\\
     %\cline{2-9}
               & Zeng II   &6.04e-4 &3.22e-4 &1.65e-4 &8.38e-5 &4.22e-5 &2.12e-5 & 1.00 (1.50)\\
     &  {  DPG} & 3.97e-1 &3.06e-1 &2.31e-1 &1.72e-1 &1.25e-1 &8.71e-2 & 0.47 ($--$)\\
      \hline
                & BE     &1.13e-2 &5.60e-3 &2.78e-3 &1.39e-3 &6.92e-4 &3.46e-4  & 1.00 (1.00) \\
     %\cline{2-9}
               & SBD     &1.63e-3 &3.79e-4 &9.15e-5 &2.25e-5 &6.56e-6 &1.37e-6 & 2.02 (2.00) \\
     %\cline{2-9}
     $0.9$      &Lin-Xu   &1.05e-2 &5.00e-3 &2.39e-3 &1.15e-3 &5.55e-4 &2.69e-4 & 1.05 (1.10)\\
     %\cline{2-9}
               & Zeng I   &1.39e-1 &8.39e-2 &4.46e-2 &1.74e-2 &3.15e-3 &1.30e-4 & $--$ (1.10)\\
     %\cline{2-9}
               & Zeng II  &2.36e-2 &1.97e-3 &9.74e-5 &5.56e-5 &2.97e-5 &1.54e-5 & 0.93 (1.10)\\
      &  {  DPG} &1.94e-1 &1.30e-1 &8.21e-2 &4.41e-2 &1.58e-2 &1.73e-3 & $--$ ($--$)\\
     \hline
     \end{tabular}
\end{center}
\end{table}

If the mesh size $h$ is small and the number $N$ of time steps is fixed,
then by Theorems \ref{thm:err-homo-fully-BE}(i) and
\ref{thm:err-homo-fully-SBD}(i), for both BE and SBD schemes there holds for $t_N\to0$
\begin{equation}\label{eqn:decay}
 \| U_h^N-u(t_N)\|_{L^2\II} \le C t_N^{{ q\al}/{2}} N^{-1} \| v \|_{\dH q}.
\end{equation}
In Table \ref{tab:singular_subdiff} and Fig. \ref{fig:singular_subdiff} we show the $L^2$-norm
of the error for cases (a) and (b), for fixed $N=10$ and $t_N\to 0$ with $\al=0.5$. In the
table, the \texttt{rate} (with respect to $t_N$, for fixed $N$ only) is computed from \eqref{eqn:decay},
and the theoretical decay rate is $t_N^{q\alpha/2}$. In the smooth case (a), the temporal
error decreases like  $O(t_N^{1/2})$, whereas in the nonsmooth case (b), it decays like $O(t_N^{1/8})$.
Note that in case (b), the initial data $v\in \dH {1/2-\ep}$ for any $\ep>0$, the formula
\eqref{eqn:decay} predicts an error decay rate $O(t^{\al/4})=O(t^{1/8})$. Hence the empirical rates
in Table \ref{tab:singular_subdiff} and Fig. \ref{fig:singular_subdiff} agree well with the theoretical
predictions, thereby fully confirming the factor $t_n^{q\alpha/2-1}$
in Theorem \ref{thm:err-homo-fully-BE}(i) (and likewise the factor $t_n^{q\alpha/2-2}$ in Theorem
\ref{thm:err-homo-fully-SBD}(i)).

\begin{table}[htb!]
\caption{The $L^2$-error for cases (a) and (b)
as $t\rightarrow0$ with $h=1/512$ and $N=10$.}
\label{tab:singular_subdiff}
\begin{center}
\vspace{-.3cm}
     \begin{tabular}{|c|c|ccccccc|c|}
     \hline
      $t$& method & 1e-3 &1e-4 & 1e-5 & 1e-6 & 1e-7 & 1e-8 &rate \\
     \hline
       (a)& BE &6.16e-3 &2.64e-3 &8.93e-4 &2.88e-4 &9.20e-5 &2.93e-5 & 0.49 (0.50) \\
       & {{SBD}} &4.52e-4 &1.55e-4 &4.98e-5 &1.59e-5 &5.04e-6 &1.60e-6 & 0.50 (0.50) \\
      \hline
       (b)& BE &5.86e-3 &4.61e-3 &3.32e-3 &2.51e-3 &1.92e-3 &1.51e-3& 0.12 (0.13) \\
       & {SBD} &5.44e-4 &3.81e-4 &2.85e-4 &2.14e-4 &1.60e-4 &1.19e-4 & 0.13 (0.13) \\
      \hline
     \end{tabular}
\end{center}
\end{table}

\begin{figure}[htb!]
  \centering
  \includegraphics[trim = 1cm .1cm 2cm 0.0cm, clip=true,width=7cm]{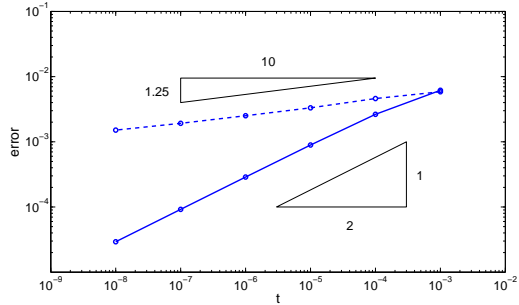}
  \caption{Numerical results for cases (a) and (b) using the BE scheme with $h=1/512$ and $N=10$,
  $\alpha=0.5$ for $t\rightarrow0$, where the solid and dashed lines stand for cases (a) and (b), respectively. }\label{fig:singular_subdiff}
\end{figure}

Last, we examine the inhomogeneous problem, i.e., case (c). The numerical results are given
in Table \ref{tab:inhomo-sub}, where the last two rows were obtained by correcting the right
hand side $f$, cf. \eqref{eqn:fullywaved-mod}. The BE scheme converges at the expected
 $O(\tau)$ rate, but the SBD scheme only converges at a rate $O(\tau^{1.18})$. The latter
is attributed to the insufficient temporal regularity of the right hand side $f$: only the first
order derivative is integrable, but not high-order ones. One can also employ the correction in \eqref{eqn:fullywaved-mod},
which seems to restore the second-order convergence, cf. the last row of Table
\ref{tab:inhomo-sub}. However, the mechanism behind this remedy is still unknown.

\begin{table}[htb!]
\caption{The $L^2$-error for case (c) at $t=0.1$, with $\alpha=0.5$ and $h=1/512$.}
\label{tab:inhomo-sub}
\begin{center}
\vspace{-.3cm}
     \begin{tabular}{|l|ccccc|c|}
     \hline
     method$\backslash N$  & $10$ &$20$ &$40$ & $80$ & $160$ &rate \\
      \hline
           BE \eqref{eqn:fullysubd}    &9.07e-4 &4.34e-5 &2.10e-5 &1.02e-5 &5.02e-6  & 1.04 (1.00) \\
          SBD \eqref{eqn:fullysub2nd}   &3.35e-6 &3.29e-6 &1.79e-6 &8.32e-7 &3.38e-7   & 1.18 ($--$)  \\
     \hline
       BE    &2.40e-4 &1.18e-4 &5.85e-5 &2.91e-5 &1.45e-5  & 1.00 (1.00) \\
      SBD    &1.99e-5 &4.61e-6 &1.11e-6 &2.68e-7 &6.32e-8   & 2.06 (2.00) \\
     \hline
     \end{tabular}
\end{center}
\end{table}

\subsection{Diffusion-wave equation}
We consider the following four examples for the diffusion wave equation
(with $\ep\in(0,1/2)$):
\begin{itemize}
\item[(d)] $v=xy(1-x)(1-y)\in \dH2$, $b=0$ and $f=0$
\item[(e)] $v= \chi_{(0,1/2]\times(0,1)}(x,y) \in \dH {{1/2}-\epsilon}$, $b=0$ and $f=0$.
\item[(f)] $v=0$, $b=\chi_{(0,1/2]\times(0,1)}(x,y)\in \dH{{1/2}-\epsilon}$, $f=0$.
\item[(g)] $v=0$, $b=0$, and $f=(1+t^{0.2})\chi_{(0,1/2]\times(0,1)}(x,y) $.
\end{itemize}
\paragraph{Numerical results for examples (d) and (e)}
First we briefly examine the convergence of the semidiscrete Galerkin scheme. The numerical results
for cases (d) and (e) are shown in Fig. \ref{fig:diffwave_smooth_space} and Table
\ref{tab:diffwave-nonsmooth-space}, respectively. We observe a convergence rate $O(h^2)$ and $O(h)$
in the $L^2$- and $H^1$-norm, respectively, for both smooth and nonsmooth data. For nonsmooth data,
the error deteriorates as $t$ approaches zero, due to the weak singularity of the solution at $t$
close to zero, cf. Theorem \ref{thm:reg-homo-space}, which we examine more closely next by verifying
the prefactors in Theorem \ref{thm:err-inhomo-semi}. For the smooth case
(d), the error essentially stays unchanged with time $t$, whereas for the nonsmooth case (e)
it deteriorates like $O(t^{-0.83})$ as $t\rightarrow 0$, cf. Table \ref{tab:singular_diffwave_space}.
These observations agree well with the theory: by Theorem \ref{thm:err-homo-semi},
as $t\to0$, there holds for any $\ep\in(0,1/2)$
\begin{equation*}
  \| u_h(t)-u(t) \|_{L^2\II}\le ct^{-\al(3+2\ep)/4}h^2\| v \|_{\dH {1/2-\ep}}.
\end{equation*}
Hence, the numerical results fully confirm the error estimates in Theorem \ref{thm:err-homo-semi}.

\begin{figure}[htb!]
  \centering
  \includegraphics[trim = 1cm .1cm 1cm 0.0cm, clip=true,width=10cm]{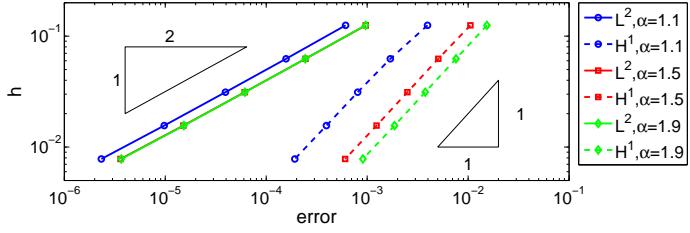}
  \caption{The convergence of the semidiscrete Galerkin scheme for case (d) at $t=0.1$, computed
  with the SBD scheme with $N=1000$.}
\label{fig:diffwave_smooth_space}
\end{figure}

\begin{table}[hbt!]
\caption{Numerical results for case (e): $\al=1.5$, $h = 2^{-k}$, $N=1000$.
}\label{tab:diffwave-nonsmooth-space}
\begin{center}
\vspace{-.3cm}
\begin{tabular}{@{}|c|c|ccccc|c|@{}}
     \hline
      $t$ & $k$ & $4$ & $5$ & $6$ &$7$ & $8$ & rate\\
     \hline
     {$0.1$} & $L^2$-norm  & 1.51e-1 &4.00e-3  &1.00e-3  &2.40e-4  &4.83e-5  &  2.05 (2.00) \\
     & $H^1$-norm            & 3.27e-1 & 1.40e-1 & 6.51e-2 & 3.11e-2 & 1.39e-2 & 1.10 (1.00) \\
     \hline
     {$0.01$} & $L^2$-norm  & 5.71e-2 &2.74e-2  &9.29e-3 &2.44e-3  &5.07e-4  &  1.92 (2.00) \\
     & $H^1$-norm           & 3.06e0 & 2.42e0 & 1.31e0 & 6.94e-1 & 2.55e-1 & 1.04 (1.00)  \\
     \hline
     $0.005$ & $L^2$-norm   & 9.87e-2 & 4.32e-2 &1.64e-2  &4.94e-3  &1.07e-3 & 1.78 (2.00)\\
     & $H^1$-norm           & 6.86e0 & 4.60e0 & 2.87e0 & 1.37e0 & 5.72e-1  & 1.00 (1.00)\\
     \hline
     \end{tabular}
\end{center}
\end{table}

\begin{table}[htb!]
\caption{The $L^2$-error for cases (d)-(f) with $\alpha=1.1$:
 $t\rightarrow0$, $h=2^{-7}$, $N=10^{3}$.}
\label{tab:singular_diffwave_space}
\begin{center}
\vspace{-.3cm}
     \begin{tabular}{|c|cccccc|c|}
     \hline
      $t$ & $1$ &1e-1 & 1e-2 & 1e-3 & 1e-4 & 1e-5 &rate \\
     \hline
      (d) &1.66e-7 &3.99e-6 &5.14e-5 &1.41e-4 &9.16e-5 &8.71e-5 & 0.02 (0) \\
      \hline
      {(e)} &1.17e-7 &2.58e-6 &1.07e-4 &1.69e-4 &9.94e-4 &6.04e-3&-0.78 (-0.83) \\
      \hline
      {(f)} &3.71e-6 &1.50e-5 &3.68e-6 &2.83e-6 &1.33e-6 &5.40e-7& 0.22 (0.18) \\
      \hline
     \end{tabular}
\end{center}
\end{table}

Next we examine temporal convergence. The numerical results for case (d) are given in Table
\ref{tab:diffwave-smooth-time}. The rates $O(\tau)$ and $O(\tau^2)$ are observed for the BE
and SBD schemes, respectively, and they hold also for case (e), cf. Table
\ref{tab:diffwave-nonsmooth-time}. Hence, the proposed schemes exhibit a steady convergence for
both smooth and nonsmooth data, verifying their robustness, cf. Theorems \ref{thm:err-homo-fully-BE}(ii) %thm:err-homo-fully-BE-wave}
and \ref{thm:err-homo-fully-SBD}(ii). Note that if the spatial error is negligible, then for fixed $N$
and $t_N\rightarrow0$, \eqref{eqn:decay} holds, by Theorem \ref{thm:err-homo-fully-BE}(ii), which
allows one to verify the temporal regularity in Theorem \ref{thm:reg-homo-time}. In Table
\ref{tab:singular_diffwave_time}, we present the results for the BE scheme for $\alpha=1.1$.
The $L^2$-norm of the error decays at a rate $O(t^{1.10})$ for (d) and $O(t^{0.28})$ for
(e), respectively, as $t\to0$, which concurs with the theoretical ones, thereby confirming
the factor $t_n^{q\alpha/2-1}$ in Theorem \ref{thm:err-homo-fully-BE}(ii).

In Tables \ref{tab:diffwave-smooth-time} and \ref{tab:diffwave-nonsmooth-time}, we present also the results by
the Crank-Nicolson (CN) scheme, which converges at a rate $O(\tau^{3-\alpha})$ for $C^3$ solutions \cite{SunWu:2006}.
It achieves the desired rate in either case, even though by Theorem \ref{thm:reg-homo-time}, the solution
$u$ does not have the requisite regularity. These observations call for further analysis of the scheme.

\begin{table}[htb!]
\caption{The $L^2$-error for case (d) at $t=0.1$ with $h=1/512$.}
\label{tab:diffwave-smooth-time}
\begin{center}
\vspace{-.3cm}
     \begin{tabular}{|c|c|cccccc|c|}
     \hline
      $\al$ & $N$ &$10$ &$20$ &$40$ & $80$ & $160$ & $320$ &rate \\
     \hline
                 & BE   &2.90e-2 &1.49e-2 &7.57e-3 &3.81e-3 &1.91e-3 &9.59e-4  & 1.00 (1.00) \\
     %\cline{2-9}
      $1.1$  & SBD  &6.35e-4 &2.02e-4 &5.49e-5 &1.41e-5 &3.45e-6 &7.50e-7 & 2.06 (2.00) \\
      %\cline{2-9}
                 & CN    &3.40e-4 &7.17e-5 &1.44e-5 &2.55e-6 &2.62e-7 &1.52e-7   & 2.36(1.90) \\
     \hline
                 &BE    &4.53e-3 &2.43e-3 &1.26e-3 &6.44e-4 &3.26e-4 &1.64e-4 & 0.99 (1.00) \\
     %\cline{2-9}
      $1.5$  & SBD  &1.25e-3 &3.26e-4 &8.22e-5 &2.03e-5 &4.51e-6 &6.72e-7 & 2.06 (2.00) \\
      %\cline{2-9}
                 & CN   &1.21e-3 &4.41e-4 &1.59e-4 &5.72e-5 &2.09e-5 &7.83e-6  & 1.45(1.50) \\
     \hline
                 &BE   &9.64e-3 &4.91e-3 &2.48e-3 &1.24e-3 &6.23e-4 &3.11e-4  & 1.00 (1.00) \\
     %\cline{2-9}
      $1.9$  & SBD  &4.41e-4 &1.25e-4 &3.38e-5 &8.58e-6 &2.10e-6 &7.66e-7 & 2.00 (2.00) \\
      %\cline{2-9}
                 & CN   &8.06e-3 &3.83e-3 &1.80e-3 &8.46e-4 &3.95e-4 &1.84e-4  & 1.10 (1.10) \\
      \hline
     \end{tabular}
\end{center}
\end{table}

\begin{table}[htb!]
\vspace{-.3cm}
\caption{The $L^2$-error for case (e) at $t=0.1$ with $h=1/512$.}
\label{tab:diffwave-nonsmooth-time}
\begin{center}
\vspace{-.3cm}
     \begin{tabular}{|c|c|cccccc|c|}
     \hline
      $\al$ & $N$ &$10$ &$20$ & $40$ & $80$ & $160$ & $320$  &rate \\
     \hline
            & BE    &2.16e-2 &1.09e-2 &5.47e-3 &2.74e-3 &1.37e-3 & 6.86e-4  &1.00 (1.00) \\
     % \cline{2-9}
      $1.1$  & SBD   &3.47e-3 &8.04e-4 &1.94e-4 &4.76e-5 &1.17e-5 &2.76e-6  & 2.02 (2.00) \\
      %\cline{2-9}
      & CN               &8.66e-2 &2.62e-2 &1.47e-3 &1.69e-5 &4.21e-6 &9.67e-7  & 2.06 (1.90) \\
      \hline
       & BE   &3.82e-2 &2.10e-2 &1.11e-2 &5.76e-3 &2.93e-3 & 1.48e-3 & 1.00 (1.00) \\
     % \cline{2-9}
     $1.5$  & SBD   &1.18e-2 &2.78e-3 &6.64e-4 &1.61e-4 &3.92e-5 &9.02e-6  & 2.06 (2.00)\\
      %\cline{2-9}
      & CN           &1.00e-2 &2.66e-3 &9.90e-4 &3.60e-4 &1.29e-4 &4.58e-5 & 1.50 (1.50)\\
      \hline
       &BE                   &1.11e-1 &7.73e-2 &5.15e-2 &3.25e-2 &1.93e-2 & 1.09e-2  & 0.83 (1.00) \\
      %\cline{2-9}
      $1.9$ & SBD        &7.36e-2 &4.11e-2 &1.85e-2 &5.96e-3 &1.56e-3 & 3.88e-4 & 1.95 (2.00) \\
      %\cline{2-9}
      & CN               &6.76e-2 &4.42e-2 &2.68e-2 &1.51e-2 &7.92e-3 &3.95e-3  & 1.00 (1.10)\\
      \hline
     \end{tabular}
\end{center}
\end{table}

\begin{table}[htb!]
\caption{The $L^2$-error for cases (d)-(e) with $\alpha=1.1$:
 $t\rightarrow 0$, $h=1/512$, and $N=10$.}
\label{tab:singular_diffwave_time}
\begin{center}
\vspace{-.3cm}
     \begin{tabular}{|c|cccccc|c|}
     \hline
      $t$ & 1e0 & 1e-1 & 1e-2  & 1e-3  & 1e-4 & 1e-5 &rate \\
     \hline
      (d) &4.44e-4 &6.33e-4 &5.93e-5 &1.82e-6 &7.94e-8 &3.97e-9 &1.30 (1.10) \\
      \hline
      (e) &2.56e-4 &3.46e-3 &1.35e-3 &7.35e-4 &3.77e-4 &4.28e-5&0.32 (0.28) \\
      \hline
      (f) &3.21e-5 &1.06e-4 &6.09e-6 &3.04e-7 &1.58e-8 &5.42e-10&  1.32 (1.28) \\
      \hline
     \end{tabular}
\end{center}
\end{table}

It is known that as the fractional order $\alpha$ increases from one to two, the model
\eqref{eqn:fde} changes from a diffusion equation to a wave one \cite{Fujita:1990}.
This transition can be observed numerically: for $\alpha$ close to one, the solution is diffusive
and very smooth, whereas for $\alpha$ close to two, the plateau in the initial data $v$ is well
preserved, reflecting a ``finite'' speed of wave propagation, cf. Fig. \ref{fig:diffwave_nonsmooth_solution}.
The small oscillations in Fig. \ref{fig:diffwave_nonsmooth_solution} are not numerical artifacts: the $L^2$ projection $P_hv$
is oscillatory, and the numerical solution inherits the feature. Further, the closer is $\alpha$ to two,
the slower is the decay of the solution (for $t$ close to zero), showing again the wave feature.

\begin{figure}[h!]
\centering
\subfigure[$\al=1.1$]{
\includegraphics[trim = 1cm .1cm 1cm 0.5cm, clip=true,width=4cm]{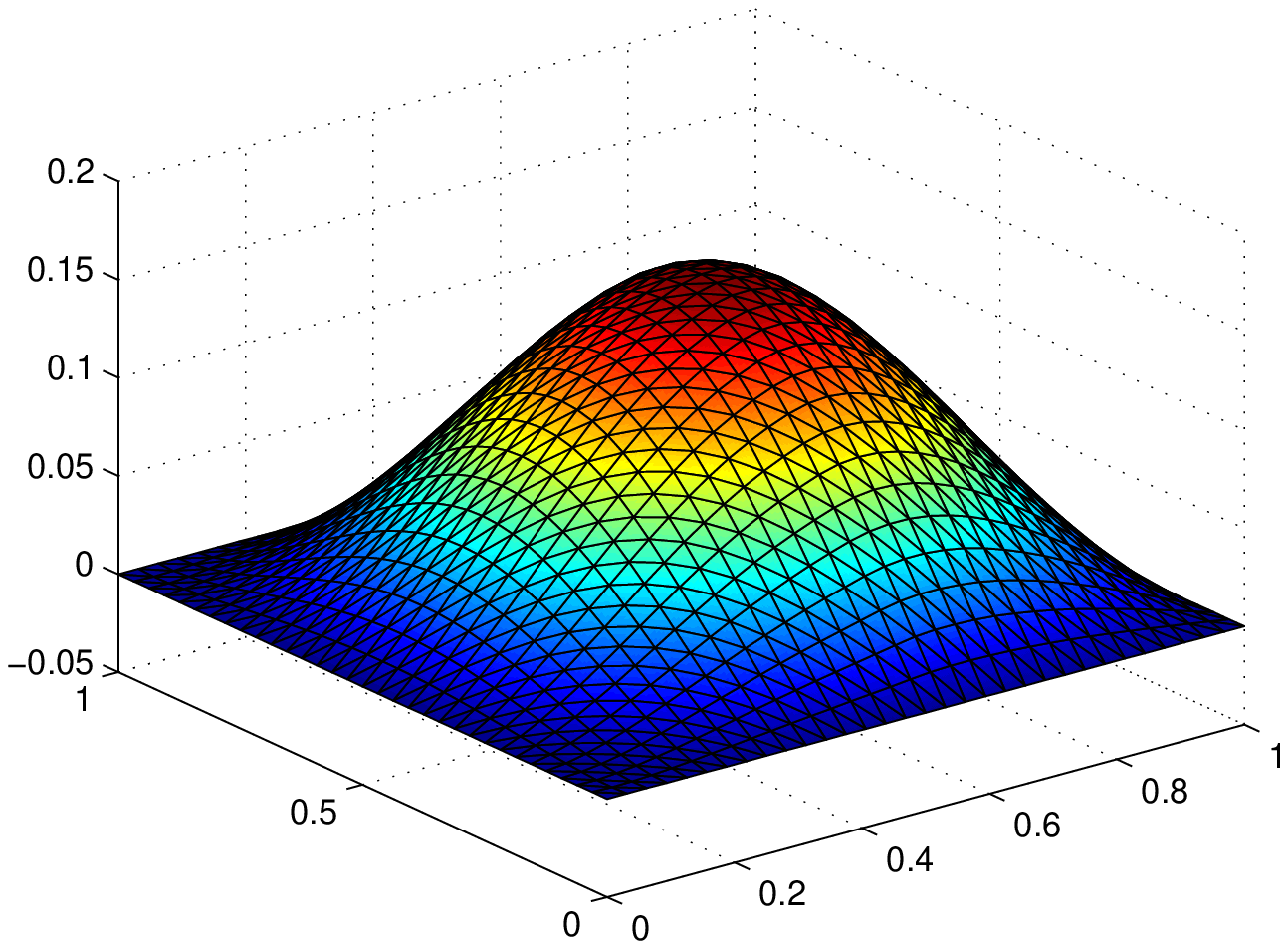}
}
\subfigure[$\al=1.5$]{
\includegraphics[trim = 1cm .1cm 1cm 0.5cm, clip=true,width=4cm]{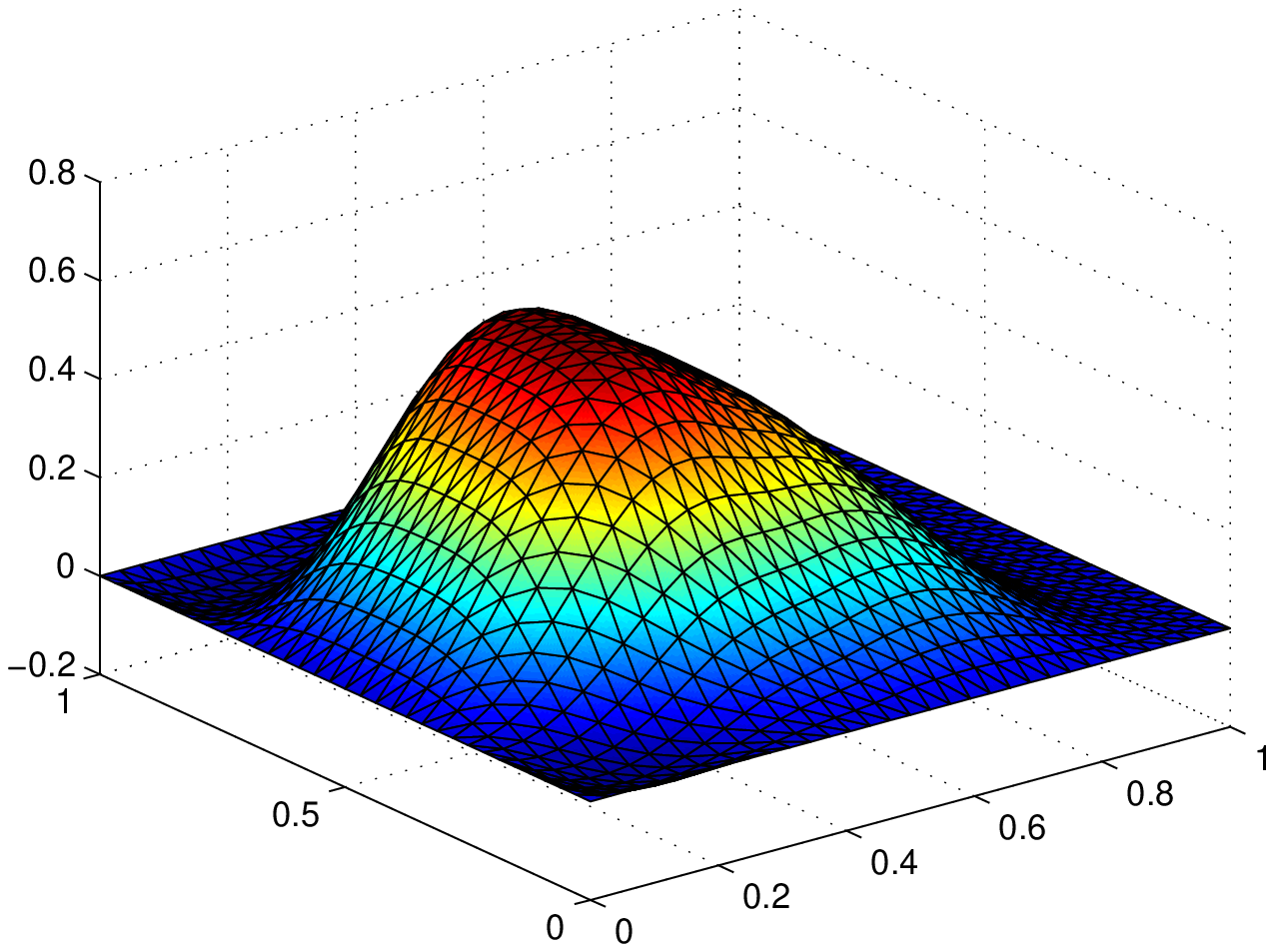}
}
\subfigure[$\al=1.9$]{
\includegraphics[trim = 1cm .1cm 1cm 0.5cm, clip=true,width=4cm]{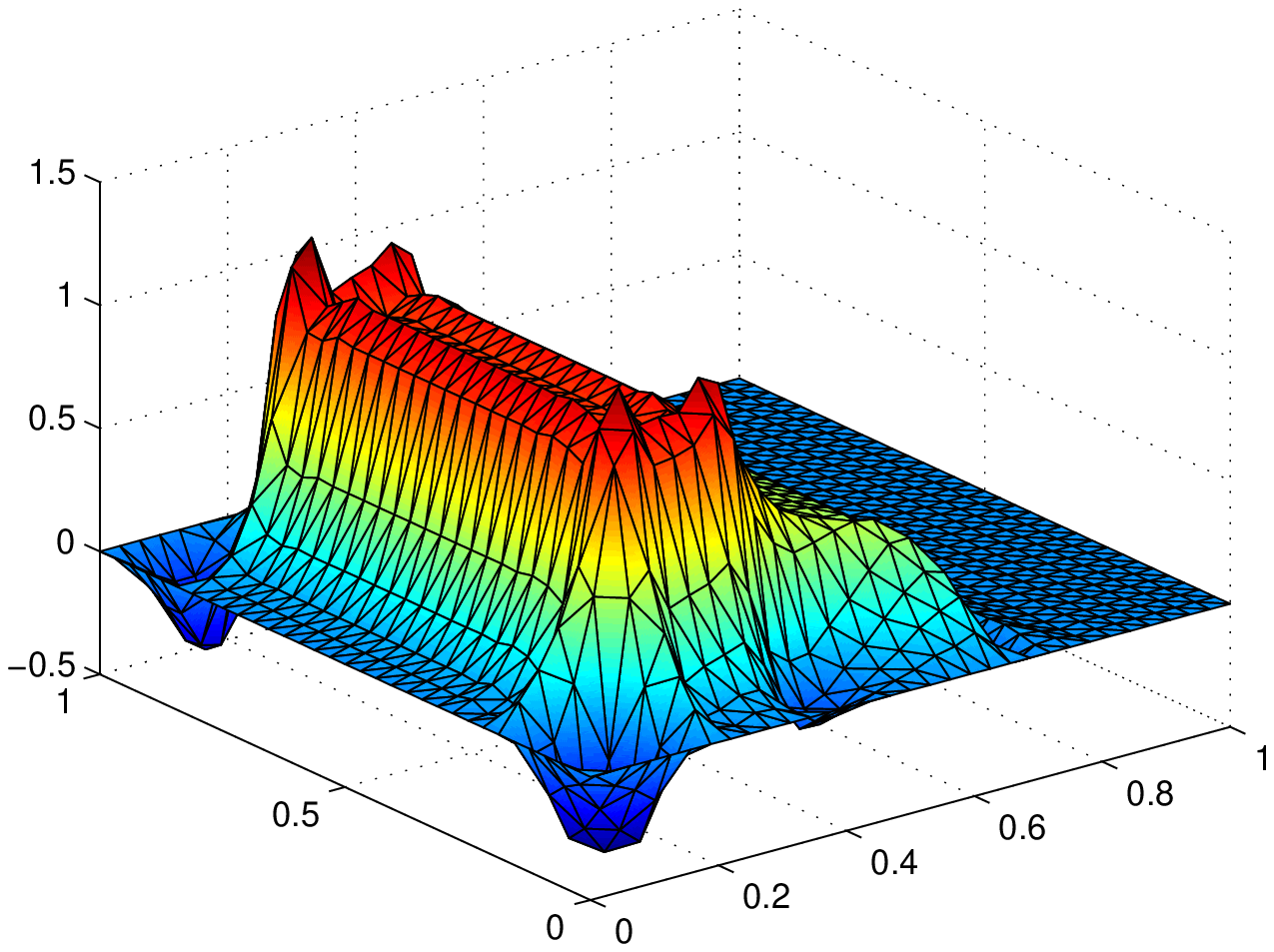}
}
  \caption{The profile of the solutions to case (e) at $t=0.1$ with three different $\alpha$ values,
  computed using the SBD scheme with $h=1/64$ and $N=160$.}
\label{fig:diffwave_nonsmooth_solution}
\end{figure}

\paragraph{Numerical results for example (f)}
Similar to cases (d) and (e), we observe the expected $O(h)$ and $O(h^2)$ convergence for the
$H^1$- and $L^2$-norm of the error, respectively, cf. Fig. \ref{fig:diffwave2-space}. An $O(\tau)$ and
$O(\tau^2)$ convergence for the BE and SBD scheme, respectively, is observed, cf. Table
\ref{tab:diffwave2-time}. To examine more closely the solution singularity, we appeal to
Theorem \ref{thm:err-homo-fully-BE}(ii) to deduce that for fixed $N$
\begin{equation*}
\begin{aligned}
   \| u(t_N)-U_h^N \|_{L^2(\Om)}
   &\leq c(N^{-1} t_N^{1+\al r/2} +h^{2} t_N^{1-{\al(2-r)}/2}) \|b\|_{\dot H^r \II}.
\end{aligned}
\end{equation*}
Hence, if the temporal error is negligible, the formula predicts a decay $O(t_N^{1-\alpha(2-r)/2})$
as $t_N\to0$. Since $b\in\dH {1/2-\ep}$ for case (f), for small $\ep\in(0,1/4)$, it predicts
$O(t_N^{0.18-0.55\ep})$ for $\alpha=1.1$, which is confirmed by the last row of Table \ref{tab:singular_diffwave_space}.
Likewise, if the spatial error is negligible, for fixed $N$, the formula predicts a decay
$O(t_N^{1+\al r/2})$ as $t_N\to0$. For $\alpha=1.1$, it predicts a rate $O(t_N^{1.28-{0.55\ep}})$,
which agrees well with the last row of Table \ref{tab:singular_diffwave_time}. These results
confirm the error estimate in Theorem \ref{thm:err-homo-fully-BE}(ii).

\begin{figure}[h!]
\includegraphics[trim = 1cm .1cm 2cm 0.5cm, clip=true,width=12cm]{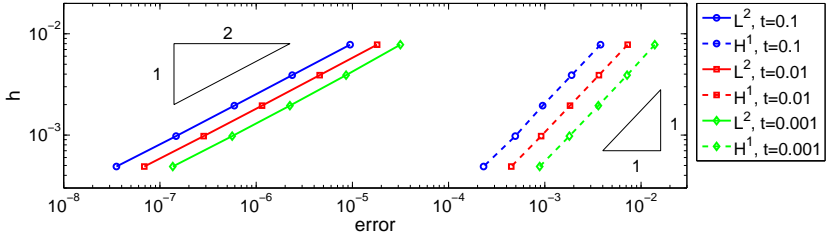}
  \caption{Error plots for case (f) with $\al=1.5$, by the SBD scheme with $N=1000$.}% at $t=0.005,~0.01, ~ 1.0$.}
\label{fig:diffwave2-space}
\end{figure}
\vspace{-.8cm}
\begin{table}[htb!]
\caption{The $L^2$-error for case (f) at $t=0.1$ with $h=1/512$.}
\label{tab:diffwave2-time}
\begin{center}
\vspace{-.3cm}{\
     \begin{tabular}{|c|c|cccccc|c|}
     \hline
        $\al\backslash N$ &method & $10$ &$20$ &$40$ & $80$ & $160$ & $320$&rate \\
     \hline
      $1.1$  &BE    &1.12e-3 &5.73e-4 &2.90e-4 &1.46e-4 &7.30e-5  &3.66e-5 & 1.00 (1.00) \\
      &SBD             & 1.07e-4 &2.54e-5 &6.20e-6 &1.53e-6 &3.76e-7 &8.83e-8   & 2.04 (2.00) \\
     \cline{1-9}
      $1.5$   &BE    &2.82e-3 &1.46e-3 &7.48e-4 &3.78e-4 &1.90e-4  &9.53e-5  & 0.99 (1.00) \\
      &SBD             & 2.37e-4 &6.47e-5 &1.65e-5 &4.11e-6 &1.01e-7 &2.30e-8  & 2.02 (2.00) \\
     \cline{1-9}
      $1.9$   &BE   &3.65e-3 &2.04e-3 &1.11e-3 &5.94e-4 &3.12e-4  &1.61e-4  & 0.94 (1.00) \\
      &SBD             & 1.02e-3 &4.13e-4 &1.42e-4 &4.06e-5 &1.05e-5 &2.93e-6   & 1.89 (2.00) \\
      \hline
     \end{tabular}}
\end{center}
\end{table}

\paragraph{Numerical results for example (g)} Now we present numerical results
for case (g) in Table \ref{tab:inhomo-wave}, where the first and last two rows are for
the basic and the corrected schemes, respectively. For the BE scheme, both variants
\eqref{eqn:fullywaved} and \eqref{eqn:fullywaved-mod} can
achieve the desired $O(\tau)$ convergence, and the errors are comparable.
However, for the SBD scheme, the basic variant converges only at a suboptimal
rate $O(\tau^{1.36})$. It concurs with Theorem \ref{thm:inhomog-fully-SBD}(i),
since the right hand side $f$ is not regular enough. The correction indeed restores
the desired convergence rate. These
observations clearly show the crucial role of proper initial correction in high-order schemes,
and in particular, an inadvertent implementation can compromise the accuracy.

\begin{table}[htb!]
\caption{The $L^2$-error for case (g) at
 $t=0.1$ with $\alpha =1.5$ and $h=1/512$.}
\label{tab:inhomo-wave}
\begin{center}
\vspace{-.3cm}
     \begin{tabular}{|c|ccccc|c|}
     \hline
       method$\backslash N$ & $10$ &$20$ &$40$ & $80$ & $160$ &rate \\
     \hline
      BE \eqref{eqn:fullywaved}    &4.20e-4 &2.20e-4 &1.14e-4 &5.81e-5 &2.97e-5   & 0.96 (1.00)  \\
       SBD (basic)   &2.40e-4 &9.25e-5 &3.66e-5 &1.45e-5 &5.43e-6  & 1.36 ($--$) \\
     \hline
      BE \eqref{eqn:fullywaved-mod}    &4.71e-4 &2.43e-4 &1.24e-4 &6.25e-5 &3.14e-5   & 0.99 (1.00) \\
       SBD \eqref{eqn:fullywave2nd-mod}   &7.64e-5 &1.92e-5 &4.77e-6 &1.17e-6 &2.73e-7  & 2.04 (2.00) \\
     \hline
     \end{tabular}
\end{center}
\end{table}

\section{Conclusions}
In this paper we develop two robust fully discrete schemes for the subdiffusion
and diffusion wave equations. The schemes employ a Galerkin finite element method
in space and the convolution quadrature generated by the backward Euler method and second-order backward
difference. We provide a complete error analysis of the schemes,
and derive optimal error estimates for both smooth and nonsmooth initial data. In
particular, the schemes achieve a first-order and second-order convergence in time.
We present extensive numerical experiments to illustrate the accuracy and robustness
of the schemes. The experimental findings fully verify the convergence
theory. Further, we compare our schemes with several existing time stepping schemes developed for
smooth solutions, and find that existing ones may be not robust with respect to data regularity.

There are several questions deserving further investigation. First, in view of
the solution singularity for nonsmooth data,
it is of much practical interest to develop time stepping schemes using a nonuniform
mesh in time and provide rigorous error analysis, including a posterior analysis. Second, our experiments indicate
that existing time stepping schemes may yield only suboptimal convergence
for nonsmooth data. This motivates revisiting these schemes for nonsmooth
data, especially sharp error estimates. Last, it is important to study more complex
models, e.g., variable coefficients in time and nonlinear models. The case of time dependent coefficients
represents one of the major challenges in applying convolution quadrature, due to a lack
of complete solution theory and loss of convolution structure.

\section*{Acknowledgements}
The authors are grateful to the anonymous referees for their constructive comments.
The research of B. Jin is partly supported by UK Engineering and Physical Sciences
Research Council grant EP/M025160/1.
\appendix
\section{The solution theory for the diffusion-wave equation}\label{app:diffwave}
In the convergence analysis, the regularity of the solution to problem \eqref{eqn:fde} plays an important role.
The solution theory for $\alpha\in(0,1)$ with nonsmooth data is now well understood
\cite{SakamotoYamamoto:2011,JinLazarovZhou:2013,JinLazarovPasciakZhou:2013a,JinLazarovPasciakZhou:2013}.
Below we describe briefly the theory for $\alpha\in(1,2)$ following these works. Using the Dirichlet eigenpairs
$\{(\lambda_j,\fy_j)\}_{j=1}^\infty$ of the negative Laplacian $-\Delta$, the solution $u$ to problem \eqref{eqn:fde}
with $1<\alpha<2$ is given by
\begin{equation*}
  u(x,t)=E(t)v + \widetilde{E}(t) b + \int_0^t\overline{E}(t-s)f(s)ds,
\end{equation*}
where the operators $E(t)$, $\widetilde{E}(t)$ and $\overline E(t)$ are given by
$ E(t)v  =\sum_{j=1}^\infty E_{\alpha,1}(-\la_j t^\al)(v, \fy_j) \fy_j(x)$,
${\widetilde E}(t) \chi = \sum_{j=1}^\infty t E_{\alpha,2}(-\la_j t^\al)\,(\chi,\fy_j) \fy_j(x)$,
$\overline{E}(t)\chi = \sum_{j=1}^\infty t^{\alpha-1}E_{\alpha,\alpha}(-\la_jt^\alpha)(\chi,\fy_j)\fy_j(x)$,
respectively, where the Mittag-Leffler function $E_{\alpha,\beta}(z)$ is defined by
$E_{\alpha,\beta}(z) = \sum_{k=0}^\infty \frac{z^k}{\Gamma(k\alpha+\beta)}$, $z\in \mathbb{C}$
\cite[pp. 42]{KilbasSrivastavaTrujillo:2006}.
It satisfies the following differentiation formula
\begin{equation}\label{eqn:mlf-diff1}
  \frac{d^m}{dt^m}t^{\beta-1}E_{\alpha,\beta}(-\lambda t^\alpha) = t^{\beta-1-m} E_{\alpha,\beta-m}(-\lambda t^\alpha)
\end{equation}
and the following asymptotics: for $\frac{\alpha\pi}{2} <\mu<\min(\pi,\alpha\pi)$ \cite[pp. 43]{KilbasSrivastavaTrujillo:2006}
\begin{equation}\label{eqn:mlf-bdd}
  |E_{\alpha,\beta}(z)|\leq c(1+|z|)^{-1}\quad\quad \mu\leq|\mathrm{arg}(z)|\leq \pi.
\end{equation}

First we give a stability result for the homogeneous problem.
\begin{theorem}\label{thm:reg-homo-space}
The solution $u(t)$ to problem \eqref{eqn:fde} with $f\equiv 0$
satisfies for $t>0$
\begin{equation*}
 \|(\Dal)^\ell u(t)\|_{\dH p} \le c(t^{-\al(\ell+(p-q)/{2})}\|v\|_{\dH q}+t^{1-\al(\ell+({p-r})/{2})}\|b\|_{\dH r}),
\end{equation*}
where for $\ell=0$, $ 0 \le q,r \le p \le 2$
and for $\ell=1$, $0 \le p \le q,r \le 2$ and $q,r \le p+2$.
\end{theorem}
\begin{proof}
First we discuss the case $\ell=0$. By the triangle inequality and \eqref{eqn:mlf-bdd},
\begin{equation*}
  \begin{aligned}
    \|E(t)v\|_{\dH p}^2 & = \sum_{j=1}^\infty \lambda_j^p|(v,\fy_j)E_{\alpha,1}(-\lambda_jt^\alpha)|^2 \leq \sum_{j=1}^\infty t^{-\alpha(p-q)}\frac{c\lambda_j^{p-q}t^{(p-q)\alpha}}{(1+\lambda_jt^\alpha)^2}\lambda_j^q(v,\fy_j)^2\\
     & \leq t^{-\alpha(p-q)}\sup_j\frac{c\lambda_j^{p-q}t^{(p-q)\alpha}}{(1+\lambda_jt^\alpha)^2}\sum_{j=1}^\infty\lambda_j^q(v,\fy_j)^2 \leq ct^{-\alpha(p-q)}\|v\|_{\dH q}^2,
  \end{aligned}
\end{equation*}
where we have used ${ (\la_jt^\al)^{p-q}}/{(1+\la_jt^\al)^2}
\leq c$ for $0\leq q\leq p\leq 2$. Similarly, one deduces
$
  \|\widetilde{E} b\|_{\dH p}^2 \leq ct^{2-\alpha(p-r)}\|b\|_{\dH r}^2.
$
Thus the assertion for $\ell=0$ follows by the triangle inequality.
Now we consider the case $\ell =1$. It follows from \eqref{eqn:mlf-bdd} that
\begin{equation*}
  \begin{aligned}
      \|\partial_t^\alpha E(t)v\|_{\dH p}^2 =&  \sum_{j=1}^\infty \la_j^{2+p}(E_{\alpha,1}(-\la_j t^\al) (v, \fy_j))^2 \\
   \leq & ct^{-\al(2+p-q)}\sum_{j=1}^\infty \frac{(\la_jt^\al)^{2+p-q}}{(1+\lambda_jt^\alpha)^2}\la_j^q(v,\fy_j)^2
    \leq ct^{-\al(2+p-q)}\|v\|_{\dH q}^2 .
  \end{aligned}
\end{equation*}
A similar estimate for $\|\partial_t^\alpha \widetilde{E}(t)b\|_{\dH p}$ holds, and this completes the proof.
\end{proof}

The next result gives temporal regularity for the homogeneous problem.
\begin{theorem}\label{thm:reg-homo-time}
If $v\in \dH q$, $b\in \dH r$, $q,r\in[0,2]$, and $f\equiv 0$, then for $m\geq1$
  \begin{equation*}
   \| \partial_t^m u \|_{L^2(\Om)} \le c t^{q\al/2-m}  \| v \|_{\dH q} + ct^{r\alpha/2-m+1}\|b\|_{\dH r}.
  \end{equation*}
\end{theorem}
\begin{proof}
By \eqref{eqn:mlf-diff1}, we have $\frac{d^m}{dt^m} E_{\alpha,1}(-\lambda t^\alpha)= -\lambda
t^{\alpha-m} E_{\alpha,\alpha+1-m}(-\lambda t^\alpha).$ Hence, by \eqref{eqn:mlf-bdd}, we deduce
\begin{equation*}
  \begin{aligned}
   \|\partial_t^m u\|_{L^2(\Omega)}^2 & = \|\sum_{j=1}^\infty \frac{d^m}{dt^m}E_{\alpha,1}(-\lambda_jt^\alpha)(v,\fy_j)\fy_j\|_{L^2\II}^2 \\
     & =\sum_{j=1}^\infty (\lambda_j t^\alpha)^{2-q} t^{q\alpha-2m}E_{\alpha,\alpha-m+1}(-\lambda_jt^\alpha)^2(v,\fy_j)^2\lambda_j^{q}\\
     & \leq ct^{q\alpha-2m}\sup_j \frac{(\lambda_jt^{\alpha})^{2-q}}{(1+\lambda_jt^\alpha)^2}\sum_{j=1}^\infty (v,\fy_j)^2\lambda_j^{q}\leq
     ct^{q\alpha-2m}\|v\|_{\dH q}.
  \end{aligned}
\end{equation*}
This shows the assertion on $v$. The assertion on $b$ follows analogously.
\end{proof}

Now we turn to inhomogeneous problems. We have the following stability result.
\begin{theorem}\label{thm:reg-inhomo-space}
For problem \eqref{eqn:fde} with $v=b=0$ and $f\in L^\infty(0,T;\dH q)$, $-1\leq q\leq1$,
there holds for any $\ep\in(0,1)$
\begin{equation*}
  \|u(t)\|_{\dH {q+2-\ep}}\leq c\ep^{-1}t^{\ep\alpha/2}\|f\|_{L^\infty(0,t;\dH q)}.
\end{equation*}
\end{theorem}
\begin{proof}
Using \eqref{eqn:mlf-bdd}, we deduce that for $ q\geq-1$ and $0\leq p-q\leq 2$,
$\|\overline E(t)\psi\|_{\dH p} \leq ct^{-1+(1+(q-p)/2)\alpha}\|\chi\|_{\dH q}.$
Consequently, the desired estimate follows by
\begin{equation*}
  \begin{aligned}
    \|u(t)\|_{\dH {q+2-\epsilon}} & = \|\int_0^t\overline E(t-s)f(s)ds\|_{\dH {q+2-\ep}}\leq \int_0^t\|\overline E(t-s)f(s)\|_{\dH {q+2-\ep}}ds\\
     & \leq c\int_0^t(t-s)^{\ep\alpha/2-1}\|f(s)\|_{\dH q}ds \leq c\ep^{-1}t^{\ep\alpha/2}\|f\|_{L^\infty(0,t;\dH q)}.
  \end{aligned}
\end{equation*}
\end{proof}

Last we state a temporal regularity result for the inhomogeneous problem.
\begin{theorem}\label{thm:reg-inhomo-time}
If $f\in W^{1,\infty}(0,T;L^2\II)$, and $v=b= 0$, then there holds
  \begin{equation}\label{eqn:new2}
   \| \partial_t^m u\|_{L^2(\Om)} \le c_T t^{\al-m}  \| f\|_{W^{m-1,\infty}(0,T;L^2\II)}, \quad m=1,2.
  \end{equation}
\end{theorem}
\begin{proof}
It follows from \eqref{eqn:mlf-diff1} and \eqref{eqn:mlf-bdd} that for $m\geq1$
\begin{equation}\label{eqn:Ebar}
\|\partial_t^m \overline E(t) \psi \|_{L^2\II} \le c t^{\al-m-1}  \|\psi\|_{L^2\II}.
\end{equation}
For the case $m=1$, using the representation $u(t)=\int_0^t\overline E(t-s)f(s)ds$, we deduce
$u'(t)=\int_0^t\overline E'(t-s)f(s)ds$, and thus
\begin{equation*}
  \begin{aligned}
    \|\partial_t\overline E(t)u\|_{L^2\II} & \leq \int_0^t\|\partial_t\overline E(t-s)f(s)\|_{L^2\II}ds
    \leq c\int_0^t (t-s)^{\alpha-2}\|f(s)\|_{L^2\II}ds \\
     & \leq ct^{\alpha-1}\|f\|_{L^\infty(0,T;L^2\II)}.
  \end{aligned}
\end{equation*}
Using the convolution relation $t(f*g)' = f*g+(tf')*g+f*(tg')$ \cite[Lemma 5.2]{McLean:2010} and \eqref{eqn:Ebar},
we deduce
\begin{equation*}
\begin{split}
  t\| \partial_t^2 u\|_{L^2(\Om)}
 & \le c\sum_{p+q\le 1}\int_0^t  \|(t-s)^p\partial_t^{p+1} \overline E(t-s)  (s^{q}f^{(q)}(s))\|_{L^2\II} \,ds\\
  & \le  c\sum_{p+q\le 1}\int_0^t  (t-s)^{\al-2} s^{q}  \|f^{(q)}(s))\|_{L^2\II} \,ds\leq c_Tt^{\alpha-1}\|f\|_{W^{1,\infty}(0,T;L^2\II)},
\end{split}
\end{equation*}
from which the desired assertion follows.
\end{proof}

Like before, the solution $u_h$ to the semidiscrete scheme \eqref{eqn:fdesemidis} is given by
\begin{equation*}
  u_h(t) = E_h(t)v_h + \widetilde{E}_h(t)b_h+\int_0^t\overline E_h(t-s)f(s)ds,
\end{equation*}
where the operators $E_h$, $\widetilde{E}_h$ and $\overline{E}_h$ are given by
$E_h(t)v_h =\sum_{j=1}^{N}E_{\alpha,1}(-\la_j^h t^\al)(v_h, \fy_j^h)$ $\fy_j^h(x)$,
${\widetilde E}_h(t) \chi_h = \sum_{j=1}^{N} t E_{\alpha,2}(-\la_j^h t^\al)(\chi_h,\fy_j^h) \fy_j^h(x)$,
${\overline E}_h(t) \chi_h = \sum_{j=1}^{N} t^{\alpha-1} E_{\alpha,\alpha}(-\la_j^h t^\al)$ $(\chi_h,\fy_j^h)\fy_j^h(x)$,
respectively, with $\{(\la_j^h,\fy_j^h)\}_{j=1}^{N}$ being the eigenpairs of the discrete Laplacian $-\Delta_h$.
Then the following discrete counterpart of Theorem \ref{thm:reg-homo-space} holds,
where $\tribar \cdot\tribar $ denotes the discrete norm defined on $X_h$, induced by $-\Delta_h$ \cite{JinLazarovZhou:2013}.
The proof is identical with that for Theorem \ref{thm:reg-homo-space} and hence omitted.
\begin{theorem} \label{lem:discrete-reg}
The solution $u_h(t)=E_h(t)v_h + \widetilde{E}_h(t)b_h$ to problem \eqref{eqn:fdesemidis} with
$f_h\equiv 0$ satisfies for $t>0$ and $0\leq q,r\leq p\leq 2$
\begin{equation*}
 \tribar u_h(t)\tribar_{\dH p} \le c(t^{-\al({p-q})/{2}}\tribar v_h\tribar_{\dH q}
 +t^{1-\al({p-r})/{2}}\tribar b_h\tribar_{\dH r}).
\end{equation*}
\end{theorem}

\bibliographystyle{siam}
\bibliography{frac}

\begin{thebibliography}{10}

\bibitem{AdamsGelhar:1992}
{\sc E~Eric Adams and Lynn~W. Gelhar}, {\em Field study of dispersion in a
  heterogeneous aquifer: 2. spatial moments analysis}, Water Res. Research, 28
  (1992), pp.~3293--3307.

\bibitem{BaleanuDiethelmScalasTrujillo:2012}
{\sc Dumitru Baleanu, Kai Diethelm, Enrico Scalas, and Juan~J. Trujillo}, {\em
  Fractional {C}alculus}, World Scientific, Hackensack, NJ, 2012.

\bibitem{BazhlekovaJinLazarovZhou:2014}
{\sc Emilia Bazhlekova, Bangti Jin, Raytcho Lazarov, and Zhi Zhou}, {\em An
  analysis of the {R}ayleigh--{S}tokes problem for a generalized second-grade
  fluid}, Numer. Math., 131 (2015), pp.~1--31.

\bibitem{ChenLiuTurnerAnh:2007}
{\sc Chang-Ming Chen, Fawang Liu, I~Turner, and V~Anh}, {\em A {F}ourier method
  for the fractional diffusion equation describing sub-diffusion}, J. Comput.
  Phys., 227 (2007), pp.~886--897.

\bibitem{ChenXuHesthaven:2015}
{\sc Feng Chen, Qinwu Xu, and Jan~S. Hesthaven}, {\em A multi-domain spectral
  method for time-fractional differential equations}, J. Comput. Phys., 293
  (2015), pp.~157--172.

\bibitem{ChenShenWang:2015}
{\sc Sheng Chen, Jie Shen, and Li-Lian Wang}, {\em Generalized {J}acobi
  functions and their applications to fractional differential equations}, Math.
  Comput.,  (2015), p.~in press.

\bibitem{CuestaLubichPlencia:2006}
{\sc Eduardo Cuesta, Christian Lubich, and Cesar Palencia}, {\em Convolution
  quadrature time discretization of fractional diffusion-wave equations}, Math.
  Comp., 75 (2006), pp.~673--696.

\bibitem{Diethelm:2004}
{\sc Kai Diethelm}, {\em The {A}nalysis of {F}ractional {D}ifferential
  {E}quations}, Lecture Notes in Mathematics, Springer, 2004.

\bibitem{FordXiaoYan:2011}
{\sc Neville~J. Ford, Jingyu Xiao, and Yubin Yan}, {\em A finite element method
  for time fractional partial differential equations}, Fract. Calc. Appl.
  Anal., 14 (2011), pp.~454--474.

\bibitem{FujitaSuzuki:1991}
{\sc Hiroshi Fujita and Takashi Suzuki}, {\em Evolution problems}, in Handbook
  of {N}umerical {A}nalysis, {V}ol.\ {II}, Handb. Numer. Anal., II,
  North-Holland, Amsterdam, 1991, pp.~789--928.

\bibitem{Fujita:1990}
{\sc Yasuhiro Fujita}, {\em Integrodifferential equation which interpolates the
  heat and the wave equation}, Osaka J. Math., 27 (1990), pp.~309--321.

\bibitem{GaoSunZhang:2014}
{\sc Guang-Hua Gao, Zhi-Zhong Sun, and Hong-Wei Zhang}, {\em A new fractional
  numerical differentiation formula to approximate the {C}aputo fractional
  derivative and its applications}, J. Comput. Phys., 259 (2014), pp.~33--50.

\bibitem{GorenfloMainardiMorettiParadisi:2002}
{\sc Rudolf Gorenflo, Francesco Mainardi, Daniele Moretti, and Paolo Paradisi},
  {\em Time fractional diffusion: a discrete random walk approach}, Nonlin.
  Dyn., 29 (2002), pp.~129--143.

\bibitem{HairerNorsettWanner:1993}
{\sc Ernst Hairer, Syvert~P. N{\o}rsett, and Gerhard Wanner}, {\em Solving
  {O}rdinary {D}ifferential {E}quations. {I}}, Springer-Verlag, Berlin,
  second~ed., 1993.
\newblock Nonstiff {P}roblems.

\bibitem{HatanoHatano:1998}
{\sc Yuko Hatano and Naomichi Hatano}, {\em Dispersive transport of ions in
  column experiments: An explanation of long-tailed profiles}, Water Res.
  Research, 34 (1998), pp.~1027--1033.

\bibitem{JinLazarovPasciakZhou:2013}
{\sc Bangti Jin, Raytcho Lazarov, Joseph Pasciak, and Zhi Zhou}, {\em Galerkin
  {FEM} for fractional order parabolic equations with initial data in
  {$H^{-s},~0\le s \le 1$}}.
\newblock LNCS 8236 (Proc. 5th Conf. Numer. Anal. Appl. (June 15-20, 2012)),
  Springer, pp. 24--37, 2013.

\bibitem{JinLazarovPasciakZhou:2014siam}
\leavevmode\vrule height 2pt depth -1.6pt width 23pt, {\em Error analysis of a
  finite element method for the space-fractional parabolic equation}, SIAM J.
  Numer. Anal., 52 (2014), pp.~2272--2294.

\bibitem{JinLazarovPasciakZhou:2013a}
\leavevmode\vrule height 2pt depth -1.6pt width 23pt, {\em Error analysis of
  semidiscrete finite element methods for inhomogeneous time-fractional
  diffusion}, IMA J. Numer. Anal., 35 (2015), pp.~561--582.

\bibitem{JinLazarovZhou:2013}
{\sc Bangti Jin, Raytcho Lazarov, and Zhi Zhou}, {\em Error estimates for a
  semidiscrete finite element method for fractional order parabolic equations},
  SIAM J. Numer. Anal., 51 (2013), pp.~445--466.

\bibitem{JinLazarovZhou:2015ima}
\leavevmode\vrule height 2pt depth -1.6pt width 23pt, {\em An analysis of the
  {L1} scheme for the subdiffusion equation with nonsmooth data}.
\newblock \emph{IMA J. Numer. Anal.}, in press, 2015.

\bibitem{JinRundell:2012}
{\sc Bangti Jin and William Rundell}, {\em An inverse problem for a
  one-dimensional time-fractional diffusion problem}, Inverse Problems, 28
  (2012), pp.~075010, 19.

\bibitem{JinRundell:2015}
\leavevmode\vrule height 2pt depth -1.6pt width 23pt, {\em A tutorial on
  inverse problems for anomalous diffusion processes}, Inverse Problems, 31
  (2015), pp.~035003, 40.

\bibitem{KilbasSrivastavaTrujillo:2006}
{\sc Anatoly~A. Kilbas, Hari~Mohan Srivastava, and Juan~J. Trujillo}, {\em
  Theory and {A}pplications of {F}ractional {D}ifferential {E}quations},
  Elsevier, Amsterdam, 2006.

\bibitem{LanglandsHenry:2005}
{\sc T.A.M. Langlands and Bruce~I. Henry}, {\em The accuracy and stability of
  an implicit solution method for the fractional diffusion equation}, J.
  Comput. Phys., 205 (2005), pp.~719--736.

\bibitem{LiDing:2014}
{\sc Changpin Li and Hengfei Ding}, {\em Higher order finite difference method
  for the reaction and anomalous-diffusion equation}, Appl. Math. Model., 38
  (2014), pp.~3802--3821.

\bibitem{LiXu:2010}
{\sc Wulan Li and Da~Xu}, {\em Finite central difference/finite element
  approximations for parabolic integro-differential equations}, Computing, 90
  (2010), pp.~89--111.

\bibitem{LiXu:2009}
{\sc Xianjuan Li and Chuanju Xu}, {\em A space-time spectral method for the
  time fractional diffusion equation}, SIAM J. Numer. Anal., 47 (2009),
  pp.~2108--2131.

\bibitem{LinLiXu:2011}
{\sc Yumin Lin, Xianjuan Li, and Chuanju Xu}, {\em Finite difference/spectral
  approximations for the fractional cable equation}, Math. Comp., 80 (2011),
  pp.~1369--1396.

\bibitem{LinXu:2007}
{\sc Yumin Lin and Chuanju Xu}, {\em Finite difference/spectral approximations
  for the time-fractional diffusion equation}, J. Comput. Phys., 225 (2007),
  pp.~1533--1552.

\bibitem{Lubich:1986}
{\sc Christian Lubich}, {\em Discretized fractional calculus}, SIAM J. Math.
  Anal., 17 (1986), pp.~704--719.

\bibitem{Lubich:1988}
\leavevmode\vrule height 2pt depth -1.6pt width 23pt, {\em Convolution
  quadrature and discretized operational calculus. {I}}, Numer. Math., 52
  (1988), pp.~129--145.

\bibitem{LubichSloanThomee:1996}
{\sc Christian Lubich, Ian~H. Sloan, and Vidar Thom{\'e}e}, {\em Nonsmooth data
  error estimates for approximations of an evolution equation with a
  positive-type memory term}, Math. Comp., 65 (1996), pp.~1--17.

\bibitem{Mainardi:1996}
{\sc Francesco Mainardi}, {\em Fractional relaxation-oscillation and fractional
  diffusion-wave phenomena}, Chaos, Solitons \& Fractals, 7 (1996),
  pp.~1461--1477.

\bibitem{Mainardi:2010}
\leavevmode\vrule height 2pt depth -1.6pt width 23pt, {\em Fractional
  {C}alculus and {W}aves in {L}inear {V}iscoelasticity}, Imperial College
  Press, London, 2010.

\bibitem{McLean:2010}
{\sc William {McLean}}, {\em Regularity of solutions to a time-fractional
  diffusion equation}, ANZIAM J., 52 (2010), pp.~123--138.

\bibitem{McLeanMustapha:2009}
{\sc William McLean and Kassem Mustapha}, {\em Convergence analysis of a
  discontinuous {G}alerkin method for a sub-diffusion equation}, Numer. Algor.,
  52 (2009), pp.~69--88.

\bibitem{McLeanThomee:2010a}
{\sc William {McLean} and Vidar Thom{\'e}e}, {\em Maximum-norm error analysis
  of a numerical solution via {L}aplace transformation and quadrature of a
  fractional-order evolution equation}, IMA J. Numer. Anal., 30 (2010),
  pp.~208--230.

\bibitem{MontrollWeiss:1965}
{\sc Elliott~W Montroll and George~H Weiss}, {\em Random walks on lattices.
  {II}}, J. Math. Phys., 6 (1965), pp.~167--181.

\bibitem{Mustapha:2014DG}
{\sc Kassem Mustapha, Basheer Abdallah, and Khaled~M. Furati}, {\em A
  discontinuous {P}etrov-{G}alerkin method for time-fractional diffusion
  equations}, SIAM J. Numer. Anal., 52 (2014), pp.~2512--2529.

\bibitem{MustaphaMcLean:2013}
{\sc Kassem Mustapha and William McLean}, {\em Superconvergence of a
  discontinuous {G}alerkin method for fractional diffusion and wave equations},
  SIAM J. Numer. Anal., 51 (2013), pp.~491--515.

\bibitem{MustaphaSchotzau:2014}
{\sc Kassem Mustapha and Dominik Sch{\"o}tzau}, {\em Well-posedness of
  {$hp$}-version discontinuous {G}alerkin methods for fractional diffusion wave
  equations}, IMA J. Numer. Anal., 34 (2014), pp.~1426--1446.

\bibitem{Nigmatulin:1986}
{\sc Raoul~R. Nigmatulin}, {\em The realization of the generalized transfer
  equation in a medium with fractal geometry}, Phys. Stat. Sol. B, 133 (1986),
  pp.~425--430.

\bibitem{OldhamSpanier:1974}
{\sc Keith~B Oldham and Jerome Spanier}, {\em The {F}ractional {C}alculus},
  Academic Press, New York, 1974.

\bibitem{Podlubny:1999}
{\sc Igor Podlubny}, {\em Fractional {D}ifferential {E}quations}, Academic
  Press, San Diego, CA, 1999.

\bibitem{SakamotoYamamoto:2011}
{\sc Kenichi Sakamoto and Masahiro Yamamoto}, {\em Initial value/boundary value
  problems for fractional diffusion-wave equations and applications to some
  inverse problems}, J. Math. Anal. Appl., 382 (2011), pp.~426--447.

\bibitem{Sanz-Serna:1988}
{\sc Jes\'{u}s~Mar\'{i}a Sanz-Serna}, {\em A numerical method for a partial
  integro-differential equation}, SIAM J. Numer. Anal., 25 (1988),
  pp.~319--327.

\bibitem{Seybold:2008}
{\sc Hansj\"{o}rg Seybold and Rudolf Hilfer}, {\em Numerical algorithm for
  calculating the generalized {M}ittag-{L}effler function}, SIAM J. Numer.
  Anal., 47 (2008/09), pp.~69--88.

\bibitem{SunWu:2006}
{\sc Zhi-Zhong Sun and Xiaonan Wu}, {\em A fully discrete scheme for a
  diffusion wave system}, Appl. Numer. Math., 56 (2006), pp.~193--209.

\bibitem{Thomee:2006}
{\sc Vidar Thom{\'e}e}, {\em Galerkin {F}inite {E}lement {M}ethods for
  {P}arabolic {P}roblems}, vol.~25 of Springer Series in Computational
  Mathematics, Springer-Verlag, Berlin, 2006.

\bibitem{Yuste:2006}
{\sc Santos~B. Yuste}, {\em Weighted average finite difference methods for
  fractional diffusion equations}, J. Comput. Phys., 216 (2006), pp.~264--274.

\bibitem{YusteAcedo:2005}
{\sc Santos~B. Yuste and Luis Acedo}, {\em An explicit finite difference method
  and a new von {N}eumann-type stability analysis for fractional diffusion
  equations}, SIAM J. Numer. Anal., 42 (2005), pp.~1862--1874.

\bibitem{ZengLiLiuTurner:2013}
{\sc Fanhai Zeng, Changpin Li, Fawang Liu, and Ian Turner}, {\em The use of
  finite difference/element approaches for solving the time-fractional
  subdiffusion equation}, SIAM J. Sci. Comput., 35 (2013), pp.~A2976--A3000.

\end{thebibliography}
\end{document}